\documentclass[11pt]{article}

\usepackage[a4paper,margin=1.05in]{geometry}
\usepackage{amsmath,amssymb,amsthm,mathtools}
\usepackage{enumitem}
\usepackage{microtype}
\usepackage{graphicx}
\usepackage{float}
\usepackage{tikz}
\usetikzlibrary{arrows.meta,positioning}
\usepackage[hidelinks]{hyperref}

\newtheorem{theorem}{Theorem}[section]
\newtheorem{proposition}[theorem]{Proposition}
\newtheorem{lemma}[theorem]{Lemma}
\newtheorem{corollary}[theorem]{Corollary}
\newtheorem{definition}[theorem]{Definition}
\newtheorem{remark}[theorem]{Remark}
\newtheorem{example}[theorem]{Example}

\DeclareMathOperator{\Fix}{Fix}
\DeclareMathOperator{\supp}{supp}
\DeclareMathOperator{\Aut}{Aut}
\DeclareMathOperator{\End}{End}
\DeclareMathOperator{\Inn}{Inn}
\DeclareMathOperator{\coker}{coker}

\DeclareMathOperator{\oc}{oc}
\DeclareMathOperator{\dc}{dc}
\DeclareMathOperator{\cfc}{cfc}
\DeclareMathOperator{\id}{id}
\DeclareMathOperator{\im}{im}

\newcommand{\RT}{\operatorname{RT}}
\newcommand{\RN}{\mathcal R}
\newcommand{\VN}{\mathcal V}
\newcommand{\MN}{\mathcal M}
\newcommand{\zetaoc}{\zeta_{\mathrm{oc}}}
\newcommand{\Z}{\mathbb Z}
\newcommand{\N}{\mathbb N}
\newcommand{\R}{\mathbb R}
\newcommand{\T}{\mathbb T}
\newcommand{\wideG}{\widehat G}
\newcommand{\phihat}{\widehat\phi}

\newcommand{\Heis}{\mathbb H}

\title{Finite-Cover Resolution Complexity of Nielsen Fixed-Point Spectra}
\author{Ahmet Selman Kaya}
\date{}

\begin{document}
\maketitle

\begin{abstract}
For a self-map $f\colon X\to X$ of a finite connected CW complex, we study
how much of the Reidemeister trace can be detected and resolved by finite
regular covers compatible with $f$.  The observed traces define a visibility
profile and three thresholds: the detection complexity $\dc(f)$, the
cancellation-free complexity $\cfc(f)$, and the observer complexity $\oc(f)$,
the least cover degree separating all essential Nielsen fixed-point classes.
We identify finite-observer indistinguishability with twisted conjugacy in the
profinite completion and determine the resulting blind subgroup.  We also
realize nonzero blind traces on a finite complex: there are maps with two
essential Nielsen classes whose observed Reidemeister trace vanishes in every
finite compatible cover.  On the fixed complex $S^1\vee S^1\vee S^2$,
observer complexity is unbounded while the induced maps on the fundamental
group and homology, the Lefschetz number, and the Nielsen number remain fixed.

For a principal torus bundle with characteristic class
$[\omega]\in H^2(\T^n;\Z^m)$ and a compatible bundle map inducing $A$ on the
base and $C$ on the fibre, in the finite-Reidemeister regime we identify the
unique maximal resolving kernel and prove
\[
  \oc(F)=|\det(I-A)|\,|\det(I-C)|
  [\Z^n:R_{I-C}(\omega)],
\]
where $R_{I-C}(\omega)$ is the radical of the characteristic class reduced
modulo $(I-C)\Z^m$.  For toral endomorphisms we compute the exact staircase
profile
\[
  \VN_{f_A}(B)=\max\{d:d\mid |\det(I-A)|,\ d\le B\}.
\]
For scalar dilations $F_{k,r}$ of integral Heisenberg nilmanifolds, the
cohomological formula gives
\[
  N(F_{k,r})=(k-1)^{2r}(k^2-1),\qquad
  \oc(F_{k,r})=(k-1)^{2r}(k^2-1)^{2r+1}.
\]
For $k=2$ we compute the full all-or-nothing profile, whose jump from one to
three visible classes occurs exactly at degree $3^{2r+1}$.  In dimension
three, the value $\VN_{F_{3,1}}(8)=5$ shows that nonabelian visibility need
not divide the Nielsen number.  We derive the corresponding observer entropy
and rational zeta function.  An
explicit toral--Heisenberg pair has identical Nielsen and Lefschetz sequences,
Nielsen zeta functions, differential eigenvalues, and topological entropy,
while the ratio of its observer-complexity sequences has exact exponential
rate $4r\log k$.
\end{abstract}

\noindent\textbf{Keywords.}
Nielsen fixed-point theory; Reidemeister trace; finite covers; principal torus bundle; characteristic class; twisted
conjugacy separability; profinite completion; Heisenberg nilmanifold; toral endomorphism; dynamical zeta function; entropy.

\noindent\textbf{2020 Mathematics Subject Classification.}
Primary 55M20; Secondary 55R15, 20E26, 20J06, 37C25, 37C30.

\section{Introduction}

The Reidemeister trace refines both the Lefschetz number and the Nielsen
number: its support is the set of essential Nielsen classes, and its
coefficients record their total local fixed-point indices; see
\cite{Jiang,Staecker}.  Passing to a finite quotient of the fundamental group
coarsens this data.  Distinct Reidemeister classes may merge, and coefficients
of opposite sign may cancel.  We ask:

\begin{quote}
How large must a finite, dynamically compatible regular cover be in order to
separate all essential Nielsen fixed-point classes of a given map?
\end{quote}

For each compatible finite quotient we define an \emph{observed Reidemeister
trace}.  Its raw class count, visible support, and surviving absolute index
mass lead to the visibility profile $\VN_f(B)$ and to three thresholds:
\[
  \dc(f)\le \cfc(f)\le \oc(f),
\]
whenever $N(f)>0$.  Here $\dc(f)$ is the first degree at which a nonzero trace
is detected, $\cfc(f)$ is the first degree preserving the full absolute index
mass, and $\oc(f)$ is the first degree separating all essential classes.
Figure~\ref{fig:resolutionhierarchy} summarizes the resulting resolution
hierarchy.

\begin{figure}[t]
\centering
\includegraphics[width=0.96\textwidth]{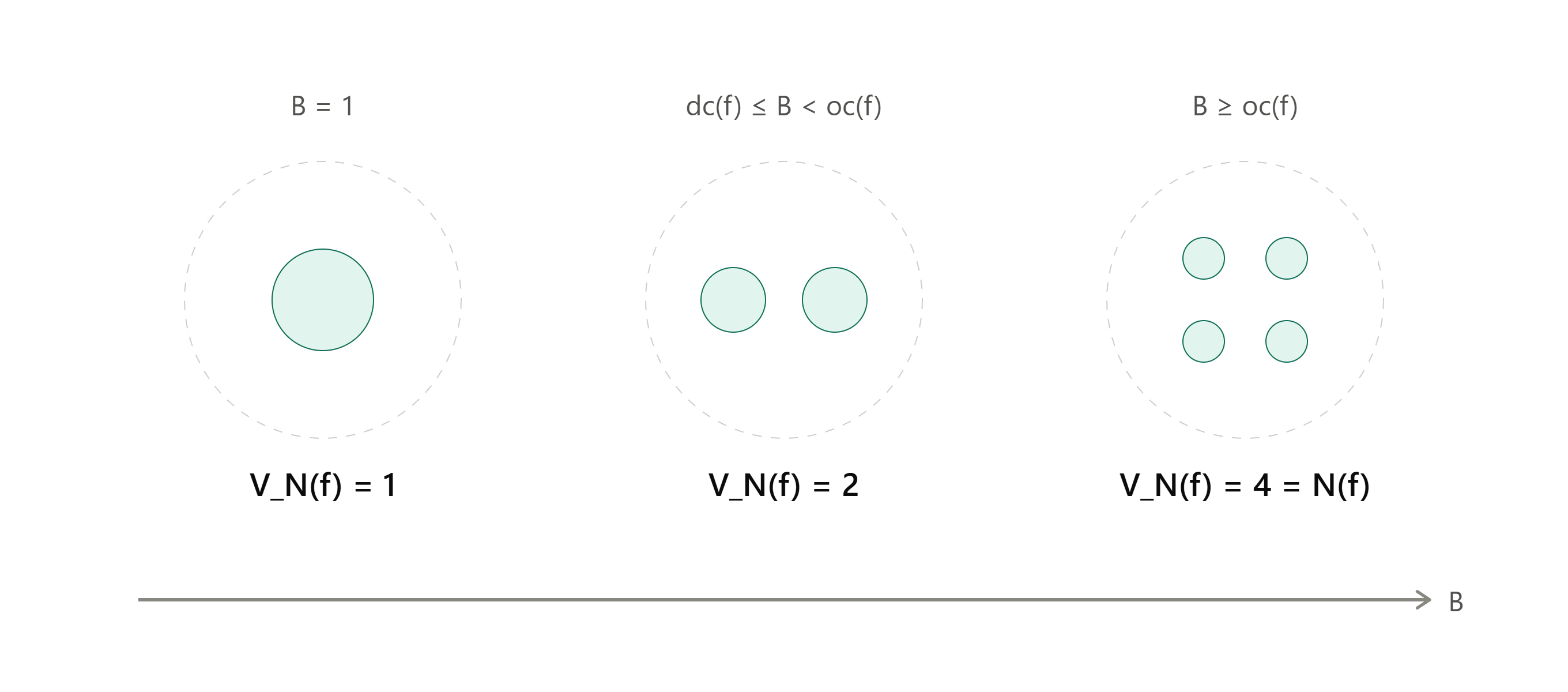}
\caption{Finite-cover resolution of the essential fixed-point spectrum as the
observer budget $B$ increases.  Low-degree covers may merge classes or cancel
their indices; once $B\ge\oc(f)$ every essential class is separated.}
\label{fig:resolutionhierarchy}
\end{figure}

Finite-quotient Reidemeister traces, subgroup-relative groupings, and
finite-cover averaging constructions are classical; see, for example,
\cite{Ponto,Staecker,LeeStaecker}.  Existence of finite quotients detecting
Reidemeister classes is related to property $RP$ and to twisted-conjugacy
separability \cite{FelTro,Tertooy}.

Effective twisted-conjugacy separability, as developed by Der\'e and
Pengitore \cite{DerePengitore}, is pairwise and asymptotic: after fixing a
word metric and an automorphism, it measures the worst-case minimal order of
a finite quotient required to separate two non-twisted-conjugate elements of
bounded word length.  The invariant studied here is of a different type.  For
a fixed self-map, it asks for the exact least order of one
$\phi$-compatible quotient whose Reidemeister-class map is injective on the
entire finite essential support of $\RT(f)$.  The detection and
cancellation-free variants also retain the integer index coefficients carried
by the trace.  In the automorphism setting covered by Der\'e--Pengitore,
their pairwise bounds imply finite simultaneous resolvability of a fixed
finite support after intersecting finitely many pair-separating kernels; they
do not determine the optimal single quotient or its exact order.  The exact
toral and central-extension formulas below solve precisely this simultaneous
optimization problem.  Our scalar Heisenberg dilations are non-surjective
endomorphisms, so the automorphism results of \cite{DerePengitore} do not
apply to them directly.

The first part of the paper develops the general structure.  Refinement makes
the observational signature monotone; finite-observer indistinguishability is
exactly indistinguishability in the profinite completion; and the blind
subgroup is the kernel of the induced map to profinite twisted-conjugacy
classes.  A general graph--sphere realization theorem shows that this blind
subgroup is topologically effective: on a finite complex one can realize two
essential Nielsen classes with opposite indices whose observed trace vanishes
in every finite compatible cover.  A second realization argument on the fixed
nonabelian complex $S^1\vee S^1\vee S^2$ produces maps for which all
observers below a prescribed degree are blind.  In particular, $\oc(f)$ is
unbounded even when the space, the induced maps on $\pi_1$ and homology, the
Lefschetz number, and the Nielsen number are fixed.

The main structural result solves the optimal-resolution problem for central
torus extensions.  Let a principal $\T^m$-bundle over $\T^n$ have
characteristic class $[\omega]\in H^2(\T^n;\Z^m)$, and let a compatible map
induce $A$ and $C$ on the base and fibre lattices.  In the finite-Reidemeister
regime, Theorem~\ref{thm:cohomologicalresolution} identifies the unique
maximal compatible resolving kernel and proves
\[
  \oc(F)=N(F)[\Z^n:R_{I-C}(\omega)],
\]
where $R_{I-C}(\omega)$ is the radical of the reduction of $[\omega]$ modulo
$(I-C)\Z^m$.  Thus the excess over the Nielsen number is determined by a
degree-two characteristic class, not by the linear fixed-point count alone.
The scalar circle-bundle formula further identifies the zero- and
maximal-penalty regimes block by block.  A heterogeneous non-Heisenberg
example shows that distinct characteristic blocks may contribute opposite
extremes to the same optimal resolving degree.

Two exact families show the range of the invariant.  For toral endomorphisms,
the complete profile is the divisor staircase
\[
  \VN_{f_A}(B)=\max\{d:d\mid |\det(I-A)|,\ d\le B\}.
\]
For scalar dilations on integral Heisenberg nilmanifolds, the characteristic
form is unimodular symplectic, and the cohomological theorem gives
\[
  \oc(F_{k,r})=(k-1)^{2r}(k^2-1)^{2r+1},
  \qquad
  \frac{\oc(F_{k,r})}{N(F_{k,r})}=(k^2-1)^{2r}.
\]
For $k=2$ the complete visibility profile has a single jump at the resolving
degree $3^{2r+1}$, while the example $F_{3,1}$ has
$\VN_{F_{3,1}}(8)=5\nmid32$.  Thus the nonabelian profile is not a toral
divisor staircase in disguise.  The iterate sequence yields an observer entropy and a rational
observer-complexity zeta function.  Theorem~\ref{thm:nondetermination} then
exhibits toral and Heisenberg systems with identical Nielsen and Lefschetz
sequences, Nielsen zeta functions, differential eigenvalues, and topological
entropy, but exponentially different finite-cover resolution sequences.

Sections~2--5 establish the observer formalism, homotopy invariance, and the
profinite reconstruction theorem.  Sections~6--7 give the realization and
delayed-visibility examples.  Section~8 computes the toral profile.
Section~\ref{sec:cohomological} proves the cohomological optimal-resolution
theorem, and Section~\ref{sec:heisenberg} develops its Heisenberg, entropy,
and zeta consequences.

\section{Reidemeister classes and finite observers}

Throughout the paper, $X$ is a finite connected CW complex, $x_0\in X$ is a
basepoint, and $f\colon X\to X$ is a cellular self-map.  We initially choose a
path from $x_0$ to $f(x_0)$ so that $f$ determines an endomorphism
\[
  \phi\in\End(G),\qquad G=\pi_1(X,x_0).
\]
The dependence on the basepoint, path, and lift is removed in
Section~\ref{sec:invariance}.

\begin{definition}
For $a,b\in G$, write $a\sim_\phi b$ if there exists $g\in G$ such that
\[
  b=g a\phi(g)^{-1}.
\]
The set of $\phi$-twisted conjugacy classes is denoted by $\RN(\phi)$.
\end{definition}

After choosing a lift of $f$ to the universal cover, the Reidemeister trace is
an element
\[
  \RT(f)=\sum_{\alpha\in\RN(\phi)} I_f(\alpha)[\alpha]
  \in \Z[\RN(\phi)]
\]
with finite support.  The coefficient $I_f(\alpha)$ is the index of the
corresponding fixed-point class.  Consequently,
\[
  N(f)=\#\supp\RT(f),
  \qquad
  L(f)=\sum_{\alpha} I_f(\alpha).
\]

\begin{definition}[Finite observer]
A finite observer for $(G,\phi)$ is a finite-index normal subgroup
$N\triangleleft G$ satisfying $\phi(N)\subseteq N$.  Set
\[
  Q_N=G/N,
\]
and let $q_N\colon G\twoheadrightarrow Q_N$ be the quotient map.  The
endomorphism induced by $\phi$ on $Q_N$ is denoted $\bar\phi_N$.
\end{definition}

Since $f_*(N)\subseteq N$, the subgroup $N$ determines an $f$-compatible finite regular
cover $X_N\to X$.  We regard the order $|Q_N|=[G:N]$ as the resolution cost of
the observer.

The quotient map induces
\[
  (q_N)_R\colon\RN(\phi)\longrightarrow\RN(\bar\phi_N),
  \qquad
  [a]_\phi\longmapsto[q_N(a)]_{\bar\phi_N}.
\]

\begin{definition}[Observed trace and visible Nielsen number]
The Reidemeister trace observed through $N$ is
\[
  \RT_N(f):=((q_N)_R)_\#\RT(f)
  \in\Z[\RN(\bar\phi_N)].
\]
Its visible Nielsen number is
\[
  V_N(f):=\#\supp\RT_N(f).
\]
\end{definition}

Explicitly, the coefficient at $\beta\in\RN(\bar\phi_N)$ is
\[
  \sum_{(q_N)_R(\alpha)=\beta} I_f(\alpha).
\]
Thus a finite observer can lose information in two distinct ways: different
classes may collide, and coefficients of opposite signs may cancel.

\begin{definition}[Finite-cover signature]
Let $E_f=\supp\RT(f)$.  For a finite observer $N$, define
\[
  R_N(f):=\#(q_N)_R(E_f),
  \qquad
  M_N(f):=\|\RT_N(f)\|_1,
\]
where $R_N(f)$ is the raw number of quotient classes hit by essential
Reidemeister classes and $M_N(f)$ is the surviving absolute index mass.  The
finite-cover signature is
\[
  \Sigma_N(f):=\bigl(R_N(f),V_N(f),M_N(f)\bigr).
\]
The collision defect and cancellation mass are
\[
  \operatorname{col}_N(f):=N(f)-R_N(f),
\]
\[
  \operatorname{can}_N(f)
  :=\frac12\bigl(\|\RT(f)\|_1-M_N(f)\bigr).
\]
\end{definition}

\begin{proposition}[Monotone observational signature]
\label{prop:signature}
If $N'\subseteq N$ are finite observers, then
\[
  R_N(f)\le R_{N'}(f),\qquad
  V_N(f)\le V_{N'}(f),\qquad
  M_N(f)\le M_{N'}(f).
\]
Consequently, both $\operatorname{col}_N(f)$ and
$\operatorname{can}_N(f)$ are nonincreasing under refinement.  Moreover,
$\operatorname{can}_N(f)$ is a nonnegative integer.
\end{proposition}

\begin{proof}
The map on Reidemeister classes for the coarse observer factors through that
for the fine observer.  Hence coarse observation can only merge raw classes,
which proves the first inequality.  The coarse observed trace is the
pushforward of the fine observed trace.  Every nonzero coarse atom has at
least one nonzero fine atom above it, proving the second inequality, while
the triangle inequality for coefficients gives the third.  The defect
statements follow immediately.  Finally, for integers $c_1,\dots,c_s$,
\[
  \sum_i|c_i|-\left|\sum_i c_i\right|
\]
is a nonnegative even integer.  Summing this identity over the observer
fibers proves the last assertion.
\end{proof}

\begin{remark}[Regular-cover traces and the present viewpoint]
Reidemeister traces for regular covering spaces under the same compatibility
condition $\phi(N)\subseteq N$ are classical; see
\cite[Section~6.7]{Ponto} and compare \cite[Section~5]{Staecker}.  The
observed trace $\RT_N(f)$ is the corresponding mod-$N$ coarsening of the
base Reidemeister trace: each quotient class receives the sum of the
fixed-point indices of the ordinary classes above it.  The new role of this
classical cover-level trace in the present paper is as an observational datum
to be optimized over all compatible finite covers, with separate control of
class collisions and index cancellation.
\end{remark}

\begin{theorem}[Cover-lift interpretation]
\label{thm:coverlift}
Let $p_N\colon X_N\to X$ be the regular cover associated with a finite
observer $N$, with deck group $Q_N$.  Choose a lift
$\widetilde f\colon\widetilde X\to\widetilde X$ to the universal cover.  For
each $\alpha\in G$, the map $\alpha\widetilde f$ descends to a lift
\[
  f_{\alpha,N}\colon X_N\to X_N
\]
of $f$.  Two such lifts $f_{\alpha,N}$ and $f_{\beta,N}$ are conjugate by a
deck transformation if and only if
\[
  (q_N)_R([\alpha]_\phi)=(q_N)_R([\beta]_\phi).
\]
Therefore the atoms of $\RT_N(f)$ are precisely deck-conjugacy classes of
lifts of $f$ to $X_N$, weighted by the sum of the indices of the ordinary
Nielsen classes producing those lifts.
\end{theorem}

\begin{proof}
For $n\in N$ one has
\[
  \alpha\widetilde f(n\widetilde x)
  =\alpha\phi(n)\widetilde f(\widetilde x).
\]
Since $\phi(n)\in N$ and $N$ is normal,
$\alpha\phi(n)\alpha^{-1}\in N$; hence $\alpha\widetilde f$ descends to
$X_N$.  If $d_{q_N(g)}$ denotes the deck transformation represented by
$g\in G$, then
\[
  d_{q_N(g)}\,f_{\alpha,N}\,d_{q_N(g)}^{-1}
  =f_{g\alpha\phi(g)^{-1},N}.
\]
Thus deck conjugacy of the descended lifts is exactly
$\bar\phi_N$-twisted conjugacy of their labels in $Q_N$.  The statement about
weights is the definition of the pushforward trace.
\end{proof}

\begin{remark}
The observed trace is not, in general, the Reidemeister trace of one chosen
lift $f_{\alpha,N}\colon X_N\to X_N$.  Rather, it is the pushforward of the
base Reidemeister trace along
\[
  \RN(\phi)\longrightarrow\RN(\bar\phi_N),
\]
or, equivalently, an index-weighted distribution over the deck-conjugacy
classes of all descended lifts.  This is the distinction between the
single-lift trace and the regular-cover grouping used in the resolution
problem; compare \cite[Section~6.7]{Ponto}.
\end{remark}

\section{Homotopy invariance and independence of choices}
\label{sec:invariance}

We now verify that the numerical invariants introduced below do not depend on
the auxiliary based choices.

\begin{lemma}[Inner change of twisting]
\label{lem:innerchange}
Let $g\in G$ and let $\phi'=\Inn(g)\circ\phi$.  Then
\[
  \tau_g\colon\RN(\phi)\longrightarrow\RN(\phi'),
  \qquad
  \tau_g([a]_\phi)=[a g^{-1}]_{\phi'}
\]
is a bijection.  Every normal $\phi$-invariant subgroup is also
$\phi'$-invariant, and for every such subgroup the induced quotient-class
diagrams commute with $\tau_g$.
\end{lemma}

\begin{proof}
If $b=h a\phi(h)^{-1}$, then
\[
  b g^{-1}
  =h(a g^{-1})\bigl(g\phi(h)g^{-1}\bigr)^{-1}
  =h(a g^{-1})\phi'(h)^{-1},
\]
so $\tau_g$ is well-defined; the analogous formula with $g^{-1}$ gives its
inverse.  If $N\triangleleft G$ and $\phi(N)\subseteq N$, then
\[
  \phi'(N)=g\phi(N)g^{-1}\subseteq N.
\]
The same formula descends to every quotient and proves commutativity.
\end{proof}

\begin{proposition}[Unbased homotopy invariance]
\label{prop:homotopyinvariance}
The observed signatures, visibility profile, detection complexity, and
observer complexity depend only on the unbased homotopy class of $f$.
Changing the basepoint, the base path, or the chosen lift only canonically
relabels the Reidemeister classes and does not change any finite quotient
order or any support-separation property.
\end{proposition}

\begin{proof}
Changing the path used to define the induced endomorphism replaces $\phi$ by
an inner twist $\Inn(g)\circ\phi$.  By Lemma~\ref{lem:innerchange}, the same
finite-index normal subgroups are admissible, their indices are unchanged,
and the Reidemeister-class sets are related by bijections commuting with all
quotient maps.  The standard change-of-lift formula for the Reidemeister trace
is exactly this relabeling.  A change of basepoint conjugates the fundamental
group by an isomorphism and carries finite-index invariant normal subgroups to
subgroups of the same index.  Finally, the Reidemeister trace is invariant
under admissible homotopy up to the canonical correspondence of Reidemeister
classes; see \cite{Jiang,Staecker}.  Hence the observed signatures, support cardinalities, detection and finite
resolvability, the visibility profile, and both minimum indices are
unchanged.
\end{proof}

\section{Profinite reconstruction and the blind subgroup}

Assume in this section that $G$ is finitely generated.  Let $\wideG$ be its
profinite completion.  The endomorphism $\phi$ extends continuously to an
endomorphism $\phihat$ of $\wideG$.

\begin{lemma}[Cofinality of endomorphism-invariant subgroups]
\label{lem:cofinal}
Let $G$ be finitely generated and let $\phi\in\End(G)$.  The finite-index
normal subgroups $N$ satisfying $\phi(N)\subseteq N$ form a cofinal system
among all finite-index normal subgroups of $G$.
\end{lemma}

\begin{proof}
Let $K\triangleleft G$ have index $n$.  Define
\[
  C_n=\bigcap_{\substack{F\text{ finite},\ |F|\le n\\
                         \psi\in\operatorname{Hom}(G,F)}}\ker\psi.
\]
There are only finitely many isomorphism types of finite groups of order at
most $n$, and since $G$ is finitely generated there are only finitely many
homomorphisms from $G$ to each of them.  Thus $C_n$ has finite index.  It is
fully invariant: if $x\in C_n$ and $\eta\in\End(G)$, then for every such
$\psi$ one has $\psi(\eta(x))=(\psi\circ\eta)(x)=e$.  In particular,
$\phi(C_n)\subseteq C_n$.  Finally, the quotient map $G\to G/K$ occurs among
the homomorphisms in the intersection, so $C_n\subseteq K$.
\end{proof}

There is a natural map
\[
  j_\phi\colon\RN(\phi)\longrightarrow\RN(\phihat)
\]
induced by the canonical homomorphism $G\to\wideG$.

The next statement is the twisted-conjugacy analogue of the classical
profinite separability criterion; compare \cite{FelTro,Tertooy}.  We record
it in order to identify precisely the observational equivalence relation used
below, rather than as a new separability theorem.

\begin{theorem}[Finite observers versus the profinite completion]
\label{thm:profinite}
For $a,b\in G$, the following are equivalent:
\begin{enumerate}[label=(\alph*)]
\item $j_\phi([a]_\phi)=j_\phi([b]_\phi)$ in $\RN(\phihat)$;
\item $(q_N)_R([a]_\phi)=(q_N)_R([b]_\phi)$ for every finite observer $N$.
\end{enumerate}
\end{theorem}

\begin{proof}
If $a$ and $b$ are $\phihat$-twisted conjugate, then their images are twisted
conjugate in every finite quotient, proving (a)$\Rightarrow$(b).

Conversely, suppose (b) holds.  For each finite observer $N$, let
\[
  C_N=\left\{h\in\wideG:
  q_N(b)=q_N(h a\phihat(h)^{-1})\right\},
\]
where $q_N$ also denotes the continuous extension $\wideG\to G/N$.
Each $C_N$ is nonempty by assumption and closed in the compact group
$\wideG$.  If $N_1,\dots,N_s$ are finite observers, then
$N=\bigcap_iN_i$ is again a finite observer and
\[
  C_N\subseteq\bigcap_i C_{N_i}.
\]
Thus the family $\{C_N\}$ has the finite intersection property.  Compactness
gives an element $h\in\bigcap_N C_N$.  By Lemma~\ref{lem:cofinal}, the
closures of finite observers form a neighborhood basis at the identity in
$\wideG$.  Hence
\[
  b=h a\phihat(h)^{-1},
\]
which proves (a).
\end{proof}

\begin{definition}[Finite-observer blind subgroup]
Define
\[
  \mathcal B_\phi
  :=\bigcap_N\ker\left(((q_N)_R)_\#:
  \Z[\RN(\phi)]\to\Z[\RN(\bar\phi_N)]\right),
\]
where the intersection is over all finite observers.
\end{definition}
Despite the terminology, $\mathcal B_\phi$ is a subgroup of the free
abelian group $\Z[\RN(\phi)]$, not a subgroup of $G$.

\begin{corollary}[Exact blind-kernel description]
\label{cor:blind}
The blind subgroup satisfies
\[
  \mathcal B_\phi=\ker (j_\phi)_\#
  =\left\langle
  [\alpha]-[\beta]:j_\phi(\alpha)=j_\phi(\beta)
  \right\rangle.
\]
Consequently,
\[
  \Z[\RN(\phi)]/\mathcal B_\phi
  \cong \Z[\im j_\phi].
\]
\end{corollary}

\begin{proof}
Theorem~\ref{thm:profinite} says that two basis elements have the same image
under every finite observer exactly when they have the same image under
$j_\phi$.  It remains to pass from basis elements to finite linear
combinations.  Let
\[
  \omega=\sum_{i=1}^s c_i[\alpha_i]
\]
have nonzero image under $(j_\phi)_\#$.  After collecting terms in the same
$j_\phi$-fiber, we may assume that the $j_\phi(\alpha_i)$ are pairwise
distinct and all $c_i$ are nonzero.  By
Theorem~\ref{thm:profinite}, every pair is separated by some finite observer;
the product of the finitely many pair-separating observers separates all
$\alpha_i$ simultaneously.  Hence the image of $\omega$ in that finite
observer is nonzero.  Therefore the intersection of the finite-observer
kernels equals $\ker(j_\phi)_\#$.  Finally, for a map of sets, the kernel of
the induced homomorphism between free abelian groups is generated by the
differences of basis elements in the same fiber.  This proves all assertions.
\end{proof}

\begin{figure}[H]
\centering
\includegraphics[width=0.92\textwidth]{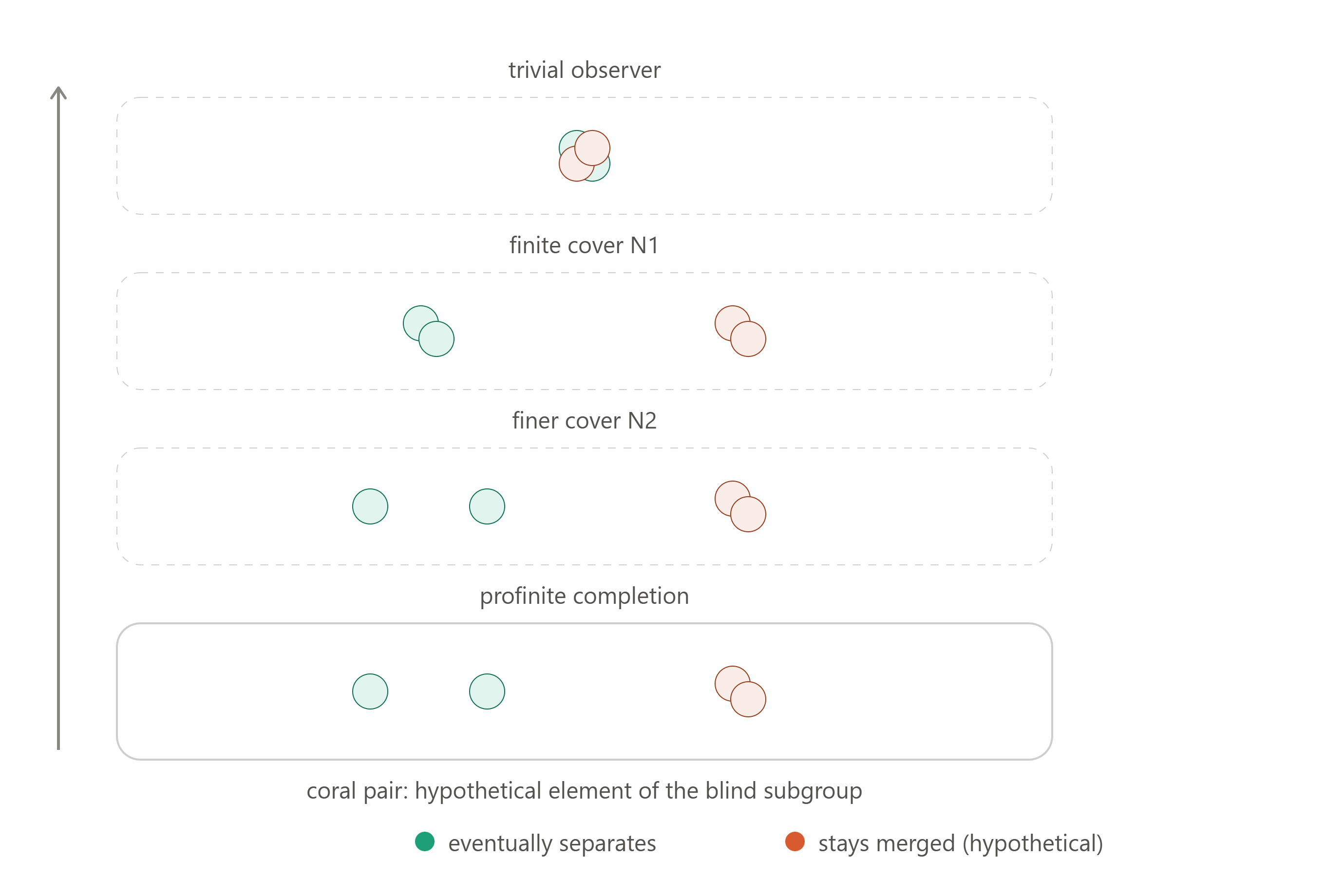}
\caption{A profinite tower of finite observers.  The green pair eventually
separates at a finite level, whereas the orange pair remains merged in every
finite compatible quotient.  The latter behavior is realized topologically
in Theorem~\ref{thm:permanentblindness}.  By Corollary~\ref{cor:blind},
persistent indistinguishability of this kind is exactly the information
carried by the finite-observer blind subgroup.}
\label{fig:profiniteblind}
\end{figure}

\begin{remark}
For a fixed endomorphism $\phi$, the condition $\mathcal B_\phi=0$ is
exactly finite-quotient separability of distinct $\phi$-twisted conjugacy
classes.  When $\phi$ is an automorphism this is the usual
$\phi$-twisted-conjugacy-separability condition.  The point of Corollary~\ref{cor:blind} is the fixed-point trace
interpretation: it identifies exactly which linear combinations of Nielsen
classes are invisible to every finite compatible cover.
\end{remark}

\section{Visibility profiles and observer complexity}
\label{sec:profile}

\begin{definition}[Visibility profile]
For $B\in\N$, define
\[
  \VN_f(B)
  :=\max\left\{V_N(f): N\triangleleft G,\ \phi(N)\subseteq N,
  \ [G:N]\le B\right\}.
\]
The trivial quotient guarantees that the set is nonempty.
\end{definition}

\begin{definition}[Observer complexity]
Let
\[
  E_f:=\supp\RT(f)\subset\RN(\phi).
\]
The observer complexity of $f$ is
\[
  \oc(f):=
  \min\left\{[G:N]:(q_N)_R|_{E_f}\text{ is injective}\right\},
\]
with $\oc(f)=\infty$ if no finite observer resolves $E_f$.  If
$N(f)\le1$, the trivial observer gives $\oc(f)=1$.
\end{definition}

\begin{definition}[Detection complexity]
The detection complexity of $f$ is
\[
  \dc(f):=
  \min\left\{[G:N]:\RT_N(f)\ne0\right\},
\]
with $\dc(f)=\infty$ if every finite observer sees zero trace.
\end{definition}

\begin{definition}[Index-mass profile and cancellation-free complexity]
For $B\in\N$, define
\[
  \MN_f(B)
  :=\max\left\{M_N(f):N\triangleleft G,\ \phi(N)\subseteq N,
  \ [G:N]\le B\right\}.
\]
The cancellation-free complexity of $f$ is
\[
  \cfc(f)
  :=\min\left\{[G:N]:M_N(f)=\|\RT(f)\|_1\right\},
\]
with $\cfc(f)=\infty$ if no finite observer preserves the full absolute
index mass.  When $\RT(f)=0$, the trivial observer gives $\cfc(f)=1$.
\end{definition}

\begin{proposition}[Detection, cancellation, and resolution]
\label{prop:threelevels}
The following hold.
\begin{enumerate}[label=(\roman*)]
\item $\MN_f(B)$ is nondecreasing and
      $0\le \MN_f(B)\le\|\RT(f)\|_1$.
\item
      \[
        \cfc(f)=\min\{B:\MN_f(B)=\|\RT(f)\|_1\}.
      \]
\item If $N(f)>0$, then, in $\N\cup\{\infty\}$,
      \[
        \dc(f)\le\cfc(f)\le\oc(f).
      \]
\item
      \[
        \cfc(f)=1
        \quad\Longleftrightarrow\quad
        |L(f)|=\|\RT(f)\|_1.
      \]
      In particular, coefficients of one sign imply $\cfc(f)=1$.
\item If $L(f)\ne0$, then $\dc(f)=1$.
\item If $N'\subseteq N$ and $\RT_N(f)\ne0$, then
      $\RT_{N'}(f)\ne0$.
\end{enumerate}
\end{proposition}

\begin{proof}
Monotonicity and the upper bound in (i) follow from
Proposition~\ref{prop:signature} and the triangle inequality; this also gives
(ii).  An observer preserving the full absolute mass has nonzero trace when
$N(f)>0$, while a resolving observer preserves every nonzero coefficient,
proving (iii).  For the trivial observer,
\[
  \RT_G(f)=L(f)[*],\qquad M_G(f)=|L(f)|,
\]
which proves (iv) and (v).  Finally, the coarse trace is the pushforward of
the fine trace, so a zero fine trace would have zero pushforward; this proves
(vi).
\end{proof}

\begin{proposition}[Basic properties]
\label{prop:basic}
The following hold.
\begin{enumerate}[label=(\roman*)]
\item $\VN_f(B)$ is nondecreasing and $0\le\VN_f(B)\le N(f)$.
\item If $N'\subseteq N$, then $V_N(f)\le V_{N'}(f)$.
\item For every observer $N$,
      \[
        V_N(f)=N(f)
        \quad\Longleftrightarrow\quad
        (q_N)_R|_{E_f}\text{ is injective}.
      \]
\item
      \[
        \oc(f)=\min\{B:\VN_f(B)=N(f)\}.
      \]
\item If $\oc(f)<\infty$, then
      \[
        N(f)\le\oc(f).
      \]
\end{enumerate}
\end{proposition}

\begin{proof}
Refinement factors the coarse pushforward through the fine one, so a coarse
observer can only merge atoms or cancel coefficients, proving (i) and (ii).
If two elements of the $N(f)$-element support collide, the image has at most
$N(f)-1$ possible support points; cancellation can only reduce this number.
This proves (iii), from which (iv) follows.  If an observer separates $N(f)$
classes, then the finite set $\RN(\bar\phi_N)$ has at least $N(f)$ elements,
and therefore $|G/N|\ge N(f)$, proving (v).
\end{proof}

\begin{theorem}[Finite resolution criterion]
\label{thm:resolutioncriterion}
The following are equivalent:
\begin{enumerate}[label=(\alph*)]
\item $\oc(f)<\infty$;
\item every two distinct elements of $E_f$ are separated by some finite
      observer;
\item $j_\phi|_{E_f}$ is injective;
\item there exists a finite observer $N$ with $V_N(f)=N(f)$.
\end{enumerate}
\end{theorem}

\begin{proof}
The equivalence of (a) and (d) follows from Proposition~\ref{prop:basic}.
Clearly (a) implies (b), while (b) and Theorem~\ref{thm:profinite} imply (c).
If (c) holds, Theorem~\ref{thm:profinite} gives, for each pair of distinct
classes in the finite set $E_f$, a finite observer separating that pair.  The
intersection of the corresponding kernels is a finite observer separating
all pairs simultaneously.  Thus (c) implies (a).
\end{proof}

\begin{proposition}[Products]
\label{prop:product}
Let $f\colon X\to X$ and $g\colon Y\to Y$ be self-maps of finite connected
CW complexes.  For observers $N$ and $M$,
\[
  \RT_{N\times M}(f\times g)=\RT_N(f)\boxtimes\RT_M(g),
\]
where $\boxtimes$ is the external product on Reidemeister classes.  Hence
\[
  R_{N\times M}(f\times g)=R_N(f)R_M(g),\qquad
  V_{N\times M}(f\times g)=V_N(f)V_M(g),
\]
\[
  M_{N\times M}(f\times g)=M_N(f)M_M(g).
\]
If $\oc(f),\oc(g)<\infty$, then
\[
  \oc(f\times g)\le\oc(f)\oc(g).
\]
\end{proposition}

\begin{proof}
Twisted-conjugacy classes and fixed-point indices are multiplicative under
Cartesian products; see \cite{Jiang}.  Thus resolving observers in the two
factors give a resolving product observer of the product index.
\end{proof}

\section{Cellular realization and permanent finite-observer blindness}

We first extend the graph--sphere construction to an arbitrary finite
presentation complex.  The following result is an adaptation of the standard
cellular Reidemeister-trace realization technique: one adjoins a $2$-sphere
and prescribes the free $\Z G$-summand in degree two.  We state the adaptation
because it converts elements of the profinite blind subgroup from Section~4
into genuine Nielsen fixed-point traces.  The new conclusion used below is
the resulting permanent finite-observer blindness, not the cellular
realization device itself; compare the classical cellular trace framework in
\cite{Jiang,Staecker}.

\begin{theorem}[Realization over a finite complex]
\label{thm:generalrealization}
Let $Y$ be a finite connected CW complex with basepoint, let
$G=\pi_1(Y)$, and put $X=Y\vee S^2$.  Let
\[
  \rho\colon\Z G\longrightarrow\Z[\RN(\id_G)]
\]
collect coefficients by ordinary conjugacy classes.  For every $b\in\Z G$
there is a based cellular map $f_b\colon X\to X$ which restricts to the
identity on $Y$ and satisfies
\[
  \RT(f_b)=\rho\bigl(\chi(Y)e+b\bigr).
\]
Consequently, every finite sum
$w\in\Z[\RN(\id_G)]$ is the Reidemeister trace of such a map.
\end{theorem}

\begin{proof}
Choose one lift of every cell of $Y$.  The cellular chain complex of the
universal cover splits as
\[
  C_*(\widetilde X)
  \cong C_*(\widetilde Y)\oplus \Z G[2],
\]
where the second summand is generated by a chosen lift of the added
$2$-sphere and has zero boundary.  Since $\widetilde X$ is simply connected,
the Hurewicz homomorphism
\[
  \pi_2(\widetilde X)\longrightarrow H_2(\widetilde X;\Z)
\]
is an isomorphism.  Hence the direct summand generated by the lifted spheres
determines a free $\Z G$-summand of $\pi_2(X)$.  The sphere component of a
based cellular self-map can therefore be chosen so that its lifted degree-two
chain map on this summand is multiplication by $b$.  Take the identity on $Y$
and this map on the sphere.

The lifted chain map is block diagonal.  On the $Y$-summand it is the
identity, whose alternating cellular trace is
\[
  \sum_j(-1)^j(\#\{j\text{-cells of }Y\})e=\chi(Y)e.
\]
The added degree-two summand contributes $b$.  The cellular Reidemeister trace
formula therefore gives
\[
  \RT(f_b)=\rho\bigl(\chi(Y)e+b\bigr).
\]
For the final assertion, choose a finite group-ring lift
$\widetilde w\in\Z G$ of $w$ and take
$b=\widetilde w-\chi(Y)e$.
\end{proof}

The delayed-visibility examples use the following special case.  For
$r\ge1$, let
\[
  X_r=\bigvee_{i=1}^r S^1_i\vee S^2,
  \qquad
  G=\pi_1(X_r)=F_r,
\]
and retain the notation $\rho\colon\Z G\to\Z[\RN(\id_G)]$.

\begin{corollary}[Graph--sphere Reidemeister-trace realization]
\label{cor:realization}
For every $w\in\Z G$, there is a based cellular map
$f_w\colon X_r\to X_r$ which is the identity on the one-skeleton and satisfies
\[
  \RT(f_w)=\rho(w).
\]
Moreover, on $H_2(X_r;\Z)\cong\Z$, the induced map is multiplication by
\[
  \varepsilon(w)+r-1,
\]
where $\varepsilon\colon\Z G\to\Z$ is the augmentation.
\end{corollary}

\begin{proof}
Apply Theorem~\ref{thm:generalrealization} with
$Y=\bigvee_{i=1}^rS_i^1$, for which $\chi(Y)=1-r$, and choose the sphere
coefficient
\[
  b=w-(1-r)e=w+(r-1)e.
\]
The ordinary degree on the added sphere is
$\varepsilon(b)=\varepsilon(w)+r-1$.
\end{proof}

\subsection{Geometric reading of the realization corollary}

Corollary~\ref{cor:realization} is a fixed-point realization statement, not only
a group-ring construction.  The coefficient of a Reidemeister class
$[g]\in\RN(\id_G)$ is the total local fixed-point index of the corresponding
Nielsen class, equivalently of the lift labelled by $g$ up to deck
conjugacy.  Passing to a compatible finite regular cover replaces the deck
label $g$ by its image in the finite deck group.  Several Nielsen classes may
therefore enter the same observed atom, and their local indices are added.

For example, take
\[
  w=e-a^{M}+b\in\Z F(a,b).
\]
The fine Reidemeister spectrum consists of three essential atoms with
coefficients $+1,-1,+1$.  A finite observer satisfying
$q(a^M)=q(e)$ but $q(b)\ne q(e)$ identifies the first two classes and cancels
their index contributions, while the third remains visible.  Figure~\ref{fig:collision-cancellation} records this fixed-point-class geometry.

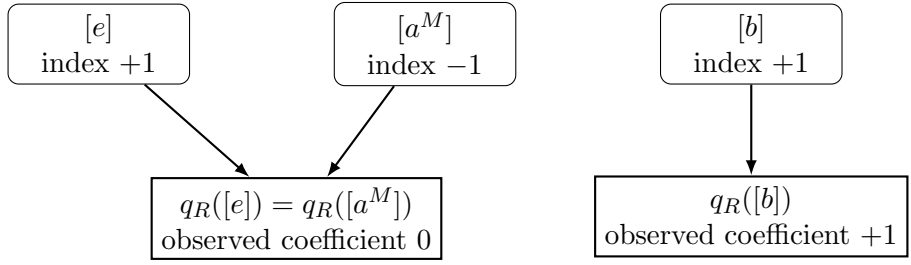
\begin{figure}[t]
\centering
\begin{tikzpicture}[
  node distance=12mm and 19mm,
  atom/.style={draw,rounded corners,minimum width=24mm,minimum height=8mm,align=center},
  observed/.style={draw,thick,minimum width=31mm,minimum height=9mm,align=center},
  arr/.style={-{Latex[length=2mm]},thick}
]
\node[atom] (e) {$[e]$\\index $+1$};
\node[atom,right=of e] (a) {$[a^M]$\\index $-1$};
\node[atom,right=of a] (b) {$[b]$\\index $+1$};
\node[observed,below=of a,xshift=-17mm] (zero) {$q_R([e])=q_R([a^M])$\\observed coefficient $0$};
\node[observed,below=of b] (plus) {$q_R([b])$\\observed coefficient $+1$};
\draw[arr] (e) -- (zero);
\draw[arr] (a) -- (zero);
\draw[arr] (b) -- (plus);
\end{tikzpicture}
\caption{Schematic fixed-point-class resolution.  The upper nodes are Nielsen
classes and the lower nodes are observed Reidemeister atoms.  The diagram
represents class identification and index addition, not the physical
locations of individual fixed points.}
\label{fig:collision-cancellation}
\end{figure}

This viewpoint will be used repeatedly below: the group-theoretic quotient is
the computational mechanism, while the quantities being merged, cancelled,
or separated are fixed-point classes and their local indices.

\begin{theorem}[Permanent finite-observer blindness]
\label{thm:permanentblindness}
There exist a finite connected CW complex $X$ and based cellular self-maps
$f_-,f_+\colon X\to X$, both inducing the identity on $\pi_1(X)$, with the
following properties:
\begin{align*}
  \RT(f_-)&=[a]-[b],&
  N(f_-)&=2,& L(f_-)&=0,\\
  \RT(f_+)&=[a]+[b],&
  N(f_+)&=2,& L(f_+)&=2,
\end{align*}
where $[a]\ne[b]$, but their images are conjugate in every finite quotient of
$\pi_1(X)$.  Consequently,
\[
  \VN_{f_-}(B)=\MN_{f_-}(B)=0
  \quad(B\ge1),
\]
\[
  \dc(f_-)=\cfc(f_-)=\oc(f_-)=\infty,
  \qquad
  0\ne[a]-[b]\in\mathcal B_{\id},
\]
whereas
\[
  \VN_{f_+}(B)=1,\qquad \MN_{f_+}(B)=2
  \quad(B\ge1),
\]
and
\[
  \dc(f_+)=\cfc(f_+)=1,
  \qquad
  \oc(f_+)=\infty.
\]
The group $\pi_1(X)$ may be chosen residually finite.
\end{theorem}

\begin{proof}
A finitely presented conjugacy-separable group has solvable conjugacy
problem; see the discussion in \cite{BogopolskiGrunewald}.  Cotton-Barratt
and Wilton construct a finitely presented residually finite group with
unsolvable conjugacy problem \cite{CottonBarrattWilton}.  Thus there exists
a finitely presented residually finite group $G$ which is not conjugacy
separable, and hence there are nonconjugate elements $a,b\in G$ whose images
are conjugate in every finite quotient of $G$.  Let $Y$ be a finite presentation complex for $G$ and put
$X=Y\vee S^2$.  Apply Theorem~\ref{thm:generalrealization} to the two trace
candidates
\[
  w_-=[a]-[b],
  \qquad
  w_+=[a]+[b].
\]
Their two support classes are distinct in $G$, so both maps have Nielsen
number two; their coefficient sums give the displayed Lefschetz numbers.

Since the induced endomorphism is the identity, every finite-index normal
subgroup is compatible.  In every finite quotient the two support classes
merge.  Thus
\[
  \RT_N(f_-)=[q(a)]-[q(b)]=0
\]
for every observer $N$, proving the claims for $f_-$.  For $f_+$ one has
\[
  \RT_N(f_+)=2[q(a)],
\]
so one atom remains visible and the full absolute index mass is preserved at
every scale, while no finite observer separates the two original classes.
The stated complexities and the nonzero blind element now follow directly
from the definitions and Corollary~\ref{cor:blind}.
\end{proof}

\begin{remark}
Residual finiteness is included to isolate the obstruction: the permanent
blindness above is not caused by a lack of finite quotients or by failure to
separate group elements.  It is specifically the failure of finite quotients
to separate the two conjugacy classes.
\end{remark}

\section{Arbitrarily delayed fixed-point visibility}

We now work on the fixed nonabelian polyhedron
\[
  X=X_2=S^1_a\vee S^1_b\vee S^2,
  \qquad
  \pi_1(X)=F(a,b).
\]
All maps constructed below induce the identity automorphism on $F(a,b)$.

\begin{remark}[Why the thresholds in this section remain finite]
\label{rem:freefinite}
Free groups are conjugacy separable; see, for example,
\cite{BogopolskiGrunewald}.  Therefore, for every map on $X_r$ inducing the
identity on $F_r$, any two distinct classes in the finite essential support
are separated by some finite quotient.  Intersecting the finitely many
pair-separating kernels gives one observer resolving the entire support.
Thus $\oc(f)<\infty$ throughout this graph--sphere setting.  The examples
below make this finite threshold arbitrarily large, while
Theorem~\ref{thm:permanentblindness} explains why a different fundamental
group is needed for $\oc=\infty$.
\end{remark}

For a prime $p$, write
\[
  M_p=\operatorname{lcm}(1,2,\dots,p-1).
\]

\begin{theorem}[Complete invisibility below a prescribed scale]
\label{thm:delayedzero}
For every prime $p$, there is a based cellular map $f_p\colon X\to X$ such
that
\[
  (f_p)_*|_{\pi_1}=\id,
  \qquad
  (f_p)_*|_{H_1}=\id,
  \qquad
  (f_p)_*|_{H_2}=\id,
\]
\[
  L(f_p)=0,
  \qquad
  N(f_p)=2,
\]
and
\[
  \VN_{f_p}(B)=
  \begin{cases}
  0,&B<p,\\[1mm]
  2,&B\ge p,
  \end{cases}
  \qquad
  \MN_{f_p}(B)=
  \begin{cases}
  0,&B<p,\\[1mm]
  2,&B\ge p.
  \end{cases}
\]
In particular,
\[
  \dc(f_p)=\cfc(f_p)=\oc(f_p)=p.
\]
\end{theorem}

\begin{proof}
Apply Corollary~\ref{cor:realization} to
\[
  w_p=e-a^{M_p}\in\Z F(a,b).
\]
Then
\[
  \RT(f_p)=[e]-[a^{M_p}].
\]
The two conjugacy classes are distinct in the free group, so $N(f_p)=2$.
Their coefficients sum to zero, hence $L(f_p)=0$.  Since $r=2$ and
$\varepsilon(w_p)=0$, Corollary~\ref{cor:realization} gives degree one on $H_2$;
the assertions on $\pi_1$ and $H_1$ follow from the identity one-skeleton.

Let $q\colon F(a,b)\twoheadrightarrow Q$ be any finite observer with
$|Q|<p$.  The order of $q(a)$ divides $|Q|$, and every positive integer below
$p$ divides $M_p$.  Therefore
\[
  q(a)^{M_p}=e_Q.
\]
The two trace atoms merge and cancel:
\[
  q_\#\RT(f_p)=[e_Q]-[e_Q]=0.
\]
Thus every observer of order below $p$ has visible Nielsen number and
surviving absolute index mass equal to zero.

On the other hand, map $a$ to a generator of $C_p$ and $b$ to the identity.
Because $p\nmid M_p$, the image of $a^{M_p}$ is nontrivial.  Hence the two
trace atoms are separated in the quotient $C_p$, so $\VN_{f_p}(p)=2$.
Propositions~\ref{prop:threelevels} and~\ref{prop:basic} now give the claimed profiles and
$\dc(f_p)=\cfc(f_p)=\oc(f_p)=p$.
\end{proof}

The preceding result shows more than unbounded complexity: all finite
observers below the threshold see no nonzero fixed-point trace at all.

\begin{theorem}[Unbounded complexity at every fixed Nielsen number]
\label{thm:fixedm}
Fix an integer $m\ge2$.  For every prime $p\ge m$, there is a based cellular
map $f_{m,p}\colon X\to X$ such that
\[
  (f_{m,p})_*|_{\pi_1}=\id,
  \qquad
  (f_{m,p})_*|_{H_1}=\id,
\]
\[
  (f_{m,p})_*|_{H_2}=(m+1)\id,
  \qquad
  L(f_{m,p})=m,
  \qquad
  N(f_{m,p})=m,
\]
and
\[
  \VN_{f_{m,p}}(B)=
  \begin{cases}
  1,&B<p,\\[1mm]
  m,&B\ge p,
  \end{cases}
  \qquad
  \MN_{f_{m,p}}(B)=m
  \quad(B\ge1).
\]
Consequently,
\[
  \cfc(f_{m,p})=1,
  \qquad
  \oc(f_{m,p})=p.
\]
\end{theorem}

\begin{proof}
Set
\[
  w_{m,p}=\sum_{j=0}^{m-1}a^{jM_p}\in\Z F(a,b)
\]
and apply Corollary~\ref{cor:realization}.  The powers
$e,a^{M_p},\dots,a^{(m-1)M_p}$ lie in distinct conjugacy classes of the free
group, so
\[
  \RT(f_{m,p})
  =\sum_{j=0}^{m-1}[a^{jM_p}]
\]
has $m$ nonzero terms and total coefficient $m$.  The homological assertions
follow from $\varepsilon(w_{m,p})=m$ and Corollary~\ref{cor:realization}.

If $|Q|<p$, then $q(a)^{M_p}=e_Q$ as in the proof of
Theorem~\ref{thm:delayedzero}; all $m$ terms merge to the single nonzero atom
$m[e_Q]$, so the visible Nielsen number is one.  Since all coefficients are
positive, no observer can reduce the absolute index mass; hence
$M_N(f_{m,p})=m$ for every observer $N$.  In $C_p$, multiplication by
$M_p$ is invertible modulo $p$, and the residues
\[
  0,M_p,2M_p,\dots,(m-1)M_p
\]
are distinct because $m\le p$.  Thus the $p$-element observer resolves all
$m$ classes.
\end{proof}

\begin{remark}[Detection and resolution may differ]
For the maps $f_{m,p}$ of Theorem~\ref{thm:fixedm}, all coefficients have
the same sign.  Hence
\[
  \dc(f_{m,p})=\cfc(f_{m,p})=1,
  \qquad
  \oc(f_{m,p})=p.
\]
Thus the full fixed-point index mass may be visible immediately, while the
individual essential classes remain unresolved until an arbitrarily large
cover degree.  In contrast, Theorem~\ref{thm:delayedzero} has
$\dc=\cfc=\oc=p$, so detection, cancellation removal, and class resolution
may either coincide or separate widely.
\end{remark}

\begin{corollary}[Independence from classical data]
\label{cor:independence}
There is no function $F\colon\N\to\N$ such that
\[
  \oc(f)\le F(N(f))
\]
for all self-maps in the present class.  More strongly, on the single finite
complex $S^1\vee S^1\vee S^2$, observer complexity is unbounded while the
induced endomorphism on $\pi_1$, the induced maps on homology, the Lefschetz
number, and the Nielsen number are all fixed.
\end{corollary}

\begin{proof}
The first claim follows from either Theorem~\ref{thm:delayedzero} or
Theorem~\ref{thm:fixedm}.  For the stronger statement use
Theorem~\ref{thm:delayedzero}: all displayed classical data are independent of
$p$, whereas $\oc(f_p)=p$.
\end{proof}

\section{Exact toral profiles and entropy thresholds}

We now compute the full visibility profile on a natural family.  Let
$A\in M_d(\Z)$ and let
\[
  f_A\colon\T^d\to\T^d,
  \qquad
  f_A([x])=[Ax].
\]

\begin{lemma}
\label{lem:abelianquotients}
If $C$ is a finite abelian group of order $D$, then for every divisor
$d\mid D$ there exists a quotient of $C$ of order $d$.
\end{lemma}

\begin{proof}
Decompose $C$ into its primary components.  A finite abelian $p$-group has a
subgroup of every order $p^j$ between $1$ and its full order, as follows
immediately by induction from its decomposition into cyclic $p$-groups.
Taking products over the primes gives a subgroup of index $d$, and the
corresponding quotient has order $d$.
\end{proof}

For a positive integer $D$, put
\[
  \delta_D(B):=\max\{d:d\mid D,\ d\le B\}.
\]

\begin{theorem}[Exact toral visibility profile]
\label{thm:toralprofile}
If
\[
  D:=|\det(I-A)|\ne0,
\]
then for every $B\ge1$,
\[
  \boxed{\VN_{f_A}(B)=\delta_D(B).}
\]
Consequently,
\[
  \dc(f_A)=1,
  \qquad
  \oc(f_A)=N(f_A)=D.
\]
\end{theorem}

\begin{proof}
For the abelian group $G=\Z^d$, the $A$-twisted conjugacy relation is
\[
  u\sim_A v
  \quad\Longleftrightarrow\quad
  v-u\in(I-A)\Z^d.
\]
Hence
\[
  \RN(A)\cong C_A:=\coker(I-A),
  \qquad |C_A|=D.
\]
The fixed points are isolated, and every Reidemeister class has local index
$\operatorname{sgn}\det(I-A)$.  Thus all $D$ classes are essential and there
is no cancellation under any quotient.

Let $N$ be any finite observer with quotient $Q=\Z^d/N$.  The induced map on
Reidemeister classes is a homomorphism
\[
  C_A\longrightarrow\coker(I-\bar A).
\]
Therefore $V_N(f_A)$ is the order of an image of $C_A$.  In particular,
$V_N(f_A)$ divides $D$ and
\[
  V_N(f_A)\le |Q|.
\]
Every observer of order at most $B$ therefore satisfies
\[
  V_N(f_A)\le\delta_D(B).
\]

Conversely, let $d\mid D$.  By Lemma~\ref{lem:abelianquotients}, choose a
quotient $C_A\twoheadrightarrow Q_d$ of order $d$, and let $N_d$ be the
preimage of its kernel under $\Z^d\to C_A$.  Since $A$ acts as the identity on
$C_A$, the subgroup $N_d$ is $A$-invariant.  The resulting observer has order
$d$, the induced endomorphism on $Q_d$ is the identity, and the class map
$C_A\to Q_d$ is onto.  Hence $V_{N_d}(f_A)=d$.  Maximizing over divisors
$d\le B$ proves the formula.

Finally, $L(f_A)=\det(I-A)\ne0$, so Proposition~\ref{prop:threelevels} gives
$\dc(f_A)=1$, while the profile formula gives $\oc(f_A)=D$.
\end{proof}

\begin{remark}[Geometric meaning of the toral profile]
By Theorem~\ref{thm:coverlift}, the $D$ Reidemeister classes are the
 deck-conjugacy classes of the relevant lifts.  A compatible cover of degree
at most $B$ can keep at most $\delta_D(B)$ of these lift classes distinct,
and a cover of that degree exists.  Thus the divisor staircase is the exact
geometric resolution curve for the essential fixed-point classes, rather
than only an arithmetic reformulation of the determinant.
\end{remark}

\begin{proposition}[The eight toral fixed points and their finite-cover fusion]
\label{prop:eighttoralfixedpoints}
Let
\[
  A=\begin{pmatrix}10&-1\\[1mm]1&0\end{pmatrix}.
\]
Then the toral endomorphism $f_A\colon\T^2\to\T^2$ has exactly eight
fixed points,
\[
  P_j=\left(\frac{j}{8},\frac{j}{8}\right),
  \qquad j=0,\ldots,7.
\]
They lie equally spaced on the diagonal circle of $\T^2$, and
\[
  \operatorname{ind}(f_A,P_j)
  =\operatorname{sgn}\det(I-A)
  =-1
\]
for every $j$.  Under addition, these points realize
\[
  \operatorname{coker}(I-A)\cong\Z/8\Z,
\]
with $P_1$ as a generator.

For each divisor $d\mid8$, the canonical quotient
\[
  \Z/8\Z\twoheadrightarrow\Z/d\Z
\]
is realized by a compatible regular cover of degree $d$.  In this
observer, the Nielsen classes represented by $P_i$ and $P_j$ have the same
observed deck-conjugacy label if and only if
\[
  i\equiv j\pmod d.
\]
Consequently, the eight fixed points are observed through the successive
partitions
\[
\begin{array}{c|l}
 d & \text{observed partition}\\ \hline
 1 & \{P_0,\ldots,P_7\},\\
 2 & \{P_0,P_2,P_4,P_6\},\ \{P_1,P_3,P_5,P_7\},\\
 4 & \{P_0,P_4\},\ \{P_1,P_5\},\ \{P_2,P_6\},\ \{P_3,P_7\},\\
 8 & \{P_0\},\ldots,\{P_7\}.
\end{array}
\]
Thus
\[
  \VN_{f_A}(B)=\delta_8(B)
\]
is the exact record of how these eight genuine fixed points become
distinguishable as the covering degree increases.
\end{proposition}

\begin{figure}[H]
\centering
\includegraphics[width=0.98\textwidth]{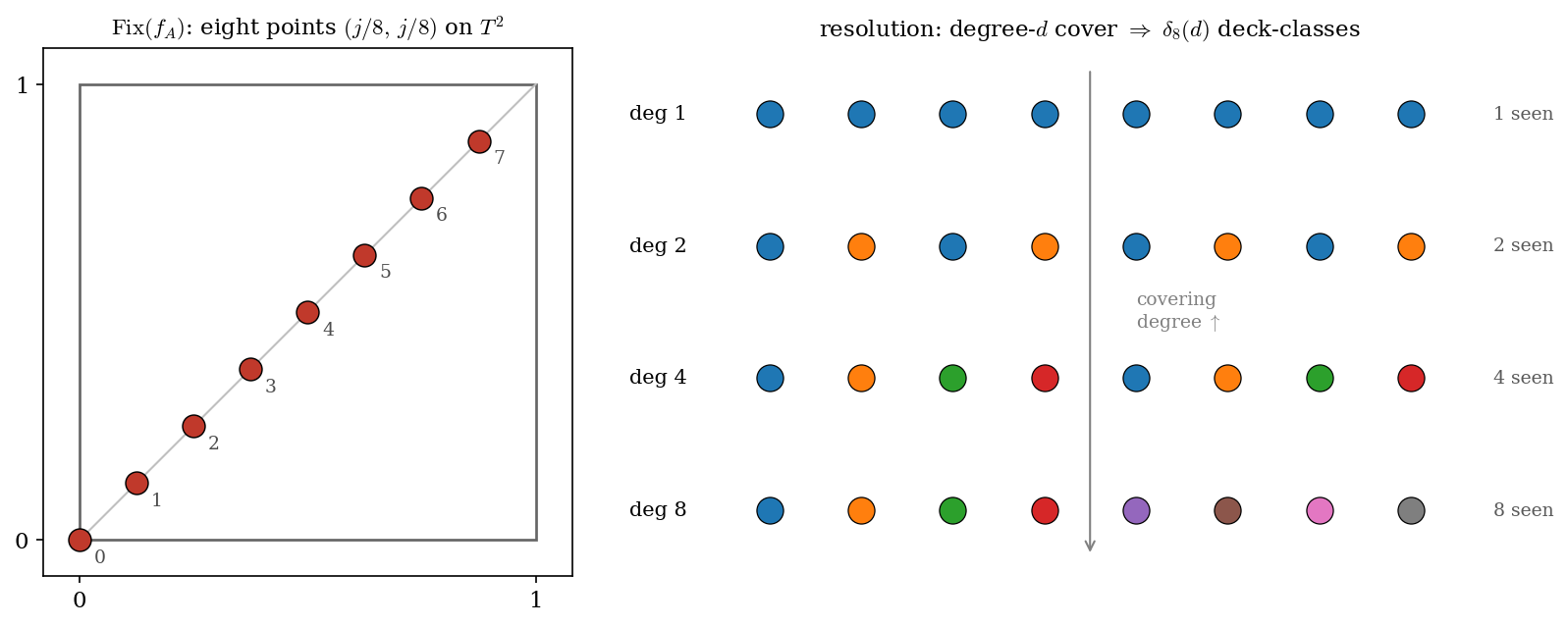}
\caption{The eight fixed points of $f_A$ on the diagonal circle of $\T^2$
and their finite-cover resolution.  The right panel shows the partitions of
the points produced by the canonical observers of degrees $1,2,4$, and $8$.
At degree $d$, two fixed-point classes have the same observed label exactly
when their indices are congruent modulo $d$.}
\label{fig:torusfixedpoints}
\end{figure}

\begin{definition}[Observer entropy]
For a self-map whose iterates have finite observer complexity, define
\[
  h_{\mathrm{oc}}(f)
  :=\limsup_{k\to\infty}
  \frac1k\log\max\{1,\oc(f^k)\}.
\]
\end{definition}

\begin{proposition}[Observer complexity--entropy identity]
\label{prop:entropy}
If $A\in GL(d,\Z)$ is hyperbolic, then for every $k\ge1$,
\[
  \VN_{f_A^k}(B)
  =\delta_{D_k}(B),
  \qquad
  \oc(f_A^k)=D_k,
  \qquad
  D_k:=|\det(I-A^k)|,
\]
and
\[
  h_{\mathrm{oc}}(f_A)=h_{\mathrm{top}}(f_A).
\]
\end{proposition}

\begin{proof}
Hyperbolicity implies that $1$ is not an eigenvalue of any $A^k$, so
Theorem~\ref{thm:toralprofile} applies to every iterate.  Let
$\lambda_1,\dots,\lambda_d$ be the complex eigenvalues of $A$, counted with
algebraic multiplicity.  Then
\[
  \frac1k\log D_k
  =\sum_{j=1}^d\frac1k\log|1-\lambda_j^k|.
\]
If $|\lambda_j|<1$, the corresponding term tends to zero.  If
$|\lambda_j|>1$, it tends to $\log|\lambda_j|$.  Hence
\[
  h_{\mathrm{oc}}(f_A)
  =\sum_{|\lambda_j|>1}\log|\lambda_j|,
\]
which is the standard entropy formula for a hyperbolic toral automorphism;
see, for example, \cite{Walters}.
\end{proof}

\begin{theorem}[Exponential resolution threshold]
\label{thm:threshold}
Let $A\in GL(d,\Z)$ be hyperbolic, set $h=h_{\mathrm{top}}(f_A)$, and let
$B_k\in\N$ satisfy
\[
  \lim_{k\to\infty}\frac1k\log B_k=\alpha.
\]
Then:
\begin{enumerate}[label=(\roman*)]
\item if $\alpha<h$, then
      \[
        \frac{\VN_{f_A^k}(B_k)}{N(f_A^k)}\longrightarrow0;
      \]
\item if $\alpha>h$, then
      \[
        \VN_{f_A^k}(B_k)=N(f_A^k)
      \]
      for all sufficiently large $k$.
\end{enumerate}
No universal conclusion at the critical value $\alpha=h$ is possible from
entropy alone; there the divisor structure of $D_k$ matters.
\end{theorem}

\begin{proof}
By Proposition~\ref{prop:entropy},
\[
  N(f_A^k)=D_k=\exp(hk+o(k)).
\]
If $\alpha<h$, then $\VN_{f_A^k}(B_k)\le B_k$, and hence
\[
  \frac{\VN_{f_A^k}(B_k)}{N(f_A^k)}
  \le\exp\bigl(-(h-\alpha)k+o(k)\bigr)\longrightarrow0.
\]
If $\alpha>h$, then $B_k\ge D_k$ eventually, so the exact profile gives full
resolution.
\end{proof}

Proposition~\ref{prop:entropy} and Theorem~\ref{thm:threshold} give a dynamical meaning to
observer complexity: entropy is the critical exponential rate of finite-cover
resources required to resolve the essential periodic fixed-point spectrum.

\begin{definition}[Observer-complexity zeta function]
\label{def:oczeta}
Suppose that $\oc(f^n)<\infty$ for every $n\ge1$.  The
observer-complexity zeta function is the formal power series
\[
  \zetaoc(f;z)
  :=\exp\left(\sum_{n\ge1}\frac{\oc(f^n)}{n}z^n\right).
\]
\end{definition}

\begin{corollary}[Toral fixed-point interpretation]
\label{cor:toralzeta}
Let $A\in GL(d,\Z)$ be hyperbolic.  Then
\[
  \oc(f_A^n)=\#\Fix(f_A^n)=|\det(I-A^n)|
  \qquad(n\ge1),
\]
and consequently
\[
  \zetaoc(f_A;z)=\zeta^{\mathrm{AM}}_{f_A}(z).
\]
In particular, $\zetaoc(f_A;z)$ is rational.
\end{corollary}

\begin{proof}
The first equality is Proposition~\ref{prop:entropy}.  Hyperbolicity implies that
$1$ is not an eigenvalue of $A^n$ for any $n$, so the fixed points of $f_A^n$
are isolated and their number is $|\det(I-A^n)|$.  The equality of zeta
functions follows from their definitions.  Rationality is classical for
toral endomorphisms; see \cite{BaakeLauPaskunas}.
\end{proof}

\section{Cohomological optimal resolution on torus bundles}
\label{sec:cohomological}

The exact Heisenberg calculation below is a special case of a more general
mechanism.  The additional finite-cover cost is controlled by the degree-two
class of a central torus extension after reduction modulo the central
Reidemeister lattice.

Let
\[
  V\cong\Z^n,
  \qquad
  Z\cong\Z^m,
\]
and let
\[
  0\longrightarrow Z\longrightarrow \Gamma
  \longrightarrow V\longrightarrow0
\]
be a central extension.  Since $V$ is free abelian, its extension class may be
viewed as an element
\[
  [\omega]\in H^2(V;Z)
  \cong H^2(\T^n;\Z^m)
  \cong \operatorname{Hom}(\Lambda^2V,Z);
\]
see, for example, \cite{Brown,Husemoller,PalaisStewart}.  Choose an integral bilinear cocycle
$\beta\colon V\times V\to Z$ representing the extension and write
\[
  \Gamma_\beta=V\times Z,
  \qquad
  (x,z)(y,w)=\bigl(x+y,z+w+\beta(x,y)\bigr).
\]
The commutator form is
\[
  \omega(x,y)=\beta(x,y)-\beta(y,x).
\]
Under the above identification, $\omega\colon\Lambda^2V\to Z$ represents the
characteristic class of the associated principal torus bundle.

Consider an endomorphism
\[
  \phi(x,z)=\bigl(Ax,Cz+\lambda(x)\bigr)
  \qquad (x\in V,\ z\in Z),
\]
where $A\in\End(V)$, $C\in\End(Z)$, and $\lambda\colon V\to Z$ satisfies
\begin{equation}
  \lambda(x+y)-\lambda(x)-\lambda(y)
  =\beta(Ax,Ay)-C\beta(x,y).
  \label{eq:cocyclecompatibility}
\end{equation}
Thus $\phi$ is a group endomorphism, and antisymmetrizing
\eqref{eq:cocyclecompatibility} gives
\begin{equation}
  C\omega(x,y)=\omega(Ax,Ay).
  \label{eq:characteristiccompatibility}
\end{equation}
Topologically, \eqref{eq:characteristiccompatibility} is the compatibility
condition $A^*[\omega]=C_*[\omega]$ for a bundle self-map.

Set
\[
  B:=I-A,
  \qquad
  D:=I-C,
\]
and assume throughout this section that
\begin{equation}
  \det B\ne0,
  \qquad
  \det D\ne0.
  \label{eq:finiteRregime}
\end{equation}
Define the full Reidemeister resolution degree of $\phi$ by
\begin{align*}
  \oc_R(\phi):=\min\bigl\{[\Gamma_\beta:N]:{}&
  N\triangleleft\Gamma_\beta,\ [\Gamma_\beta:N]<\infty,\\
  &\phi(N)\subseteq N,\text{ and }(q_N)_R
  \text{ is injective on }\RN(\phi)\bigr\}.
\end{align*}
For nilmanifold maps in the finite-Reidemeister regime all Reidemeister classes
are essential; see, for example,
\cite{DekimpeDugardein,FelLee,DeconinckDekimpe}.  Hence this quantity agrees
with the observer complexity defined in Section~5.

Reduce the characteristic form modulo the central Reidemeister lattice:
\[
  \bar\omega_D\colon\Lambda^2V\longrightarrow Z/DZ.
\]
Its radical is
\begin{equation}
  R_D(\omega)
  :=\left\{p\in V:\omega(V,p)\subseteq DZ\right\}.
  \label{eq:reducedradical}
\end{equation}
It has finite index in $V$: if $e$ is the exponent of the finite group
$Z/DZ$, then $eV\subseteq R_D(\omega)$.

\begin{lemma}[Invariance of the reduced radical]
\label{lem:radicalinvariance}
Under \eqref{eq:characteristiccompatibility},
\[
  A^{-1}R_D(\omega)=R_D(\omega).
\]
In particular, $A(R_D(\omega))\subseteq R_D(\omega)$.
\end{lemma}

\begin{proof}
Modulo $DZ$, the endomorphism $C$ is the identity.  Hence
\[
  \bar\omega_D(Ax,Ay)=\bar\omega_D(x,y)
  \qquad(x,y\in V).
\]
If $Ap\in R_D(\omega)$, then
\[
  \bar\omega_D(x,p)=\bar\omega_D(Ax,Ap)=0
  \qquad(x\in V),
\]
so $p\in R_D(\omega)$.  Thus $A^{-1}R_D(\omega)\subseteq R_D(\omega)$.
Moreover,
\[
  V/A^{-1}R_D(\omega)
  \cong (A(V)+R_D(\omega))/R_D(\omega),
\]
so
\[
  [V:A^{-1}R_D(\omega)]\le [V:R_D(\omega)].
\]
The subgroup inclusion gives the reverse inequality.  Therefore equality
holds and $A^{-1}R_D(\omega)=R_D(\omega)$.
\end{proof}

\begin{lemma}[Universal invisible subgroup]
\label{lem:universalinvisible}
For an endomorphism $\psi$ of a group $G$, set
\[
  U_\psi:=\{n\in G:g\sim_\psi gn\text{ for every }g\in G\}.
\]
Let $N\triangleleft G$ satisfy $\psi(N)\subseteq N$.  Then the quotient-class
map
\[
  (q_N)_R\colon\RN(\psi)\longrightarrow\RN(\bar\psi_N)
\]
is injective if and only if $N\subseteq U_\psi$.
\end{lemma}

\begin{proof}
If the quotient-class map is injective, then $q_N(g)=q_N(gn)$ for every
$g\in G$ and $n\in N$, so injectivity gives $g\sim_\psi gn$.  Hence
$N\subseteq U_\psi$.

Conversely, suppose $N\subseteq U_\psi$ and that $q_N(g)$ and $q_N(h)$ are
$\bar\psi_N$-twisted conjugate.  There exist $x\in G$ and $n\in N$ such that
\[
  h=xg\psi(x)^{-1}n.
\]
Since $n\in U_\psi$, right multiplication by $n$ does not change the
$\psi$-Reidemeister class of $xg\psi(x)^{-1}$.  Therefore
$g\sim_\psi h$, proving injectivity.
\end{proof}

\begin{lemma}[Twisted coordinates and the invisible subgroup]
\label{lem:centralextensioncoordinates}
For $h=(p,w)$ and $g=(x,z)$ in $\Gamma_\beta$,
\begin{equation}
 h g\phi(h)^{-1}
 =\left(
 x+Bp,
 z+Dw-\lambda(p)+\beta(p,x)-\beta(x,Ap)-\beta(Bp,Ap)
 \right).
 \label{eq:generaltwistedcoordinates}
\end{equation}
Moreover,
\begin{equation}
  U_\phi
  =\left\{
  (Bp,e):
  p\in R_D(\omega),\quad
  e+\lambda(p)+\beta(Bp,Ap)\in DZ
  \right\}.
  \label{eq:invisiblesubgroupformula}
\end{equation}
\end{lemma}

\begin{proof}
The inverse in $\Gamma_\beta$ is
\[
  (x,z)^{-1}=\bigl(-x,-z+\beta(x,x)\bigr).
\]
Substitution gives \eqref{eq:generaltwistedcoordinates}.  Now let
$n=(s,e)$.  Since
\[
  gn=\bigl(x+s,z+e+\beta(x,s)\bigr),
\]
the equality $g\sim_\phi gn$ forces $s=Bp$ for some $p\in V$.  Because
$B$ is injective under \eqref{eq:finiteRregime}, this $p$ is unique.  Comparing
the central coordinates in \eqref{eq:generaltwistedcoordinates} gives
\[
  e+\lambda(p)+\beta(Bp,Ap)+\omega(x,p)\in DZ.
\]
This holds for every $x\in V$ precisely when $p\in R_D(\omega)$ and
$e+\lambda(p)+\beta(Bp,Ap)\in DZ$, proving
\eqref{eq:invisiblesubgroupformula}.
\end{proof}

\begin{theorem}[Cohomological optimal-resolution theorem]
\label{thm:cohomologicalresolution}
Under \eqref{eq:finiteRregime}, define
\begin{equation}
  K_\phi
  :=\left\{
  (Bp,e):
  p\in R_D(\omega),\quad
  e+\lambda(p)+\beta(Bp,Ap)\in DZ
  \right\}.
  \label{eq:maximalresolvingkernel}
\end{equation}
Then $K_\phi=U_\phi$, and this subgroup is normal, finite-index, and satisfies
$\phi(K_\phi)\subseteq K_\phi$.  The quotient
$\Gamma_\beta/K_\phi$ separates all $\phi$-Reidemeister classes, and every
compatible resolving kernel $N$ satisfies
\[
  N\subseteq K_\phi.
\]
Consequently, $K_\phi$ is the unique maximal compatible resolving kernel and
\begin{equation}
  \boxed{
  \oc_R(\phi)
  =|\det(I-A)|\,|\det(I-C)|\,[V:R_D(\omega)].}
  \label{eq:cohomologicalocformula}
\end{equation}
Furthermore,
\begin{equation}
  R(\phi)=|\det(I-A)|\,|\det(I-C)|,
  \label{eq:centralextensionR}
\end{equation}
so
\begin{equation}
  \boxed{
  \frac{\oc_R(\phi)}{R(\phi)}=[V:R_D(\omega)].}
  \label{eq:cohomologicalpenalty}
\end{equation}
\end{theorem}

\begin{proof}
By Lemma~\ref{lem:centralextensioncoordinates}, $K_\phi=U_\phi$.  First note that this set is a subgroup.  If $n_1,n_2\in U_\phi$, then
$g\sim_\phi gn_1\sim_\phi gn_1n_2$ for every $g$; the inverse property follows
by applying the defining condition to $gn^{-1}$.  Since $B$ is injective, the
parameter $p$ in \eqref{eq:invisiblesubgroupformula} is unique and is additive
under multiplication.  Because $R_D(\omega)$ is a subgroup, it follows that
$K_\phi=U_\phi$ is a subgroup.

We next prove normality.  For $y=(v,a)$ and $n=(Bp,e)$, direct conjugation
gives
\[
  yny^{-1}=\bigl(Bp,e+\omega(v,Bp)\bigr).
\]
Since $Bp=p-Ap$ and Lemma~\ref{lem:radicalinvariance} gives $p,Ap\in R_D(\omega)$,
\[
  \omega(v,Bp)=\omega(v,p)-\omega(v,Ap)\in DZ.
\]
Thus conjugation preserves the defining congruence in
\eqref{eq:maximalresolvingkernel}, and $K_\phi\triangleleft\Gamma_\beta$.

For $\phi$-invariance, let $(Bp,e)\in K_\phi$.  Then
\[
  \phi(Bp,e)=\bigl(BAp,Ce+\lambda(Bp)\bigr),
\]
and $Ap\in R_D(\omega)$ by Lemma~\ref{lem:radicalinvariance}.  Applying
\eqref{eq:cocyclecompatibility} to $Bp$ and $Ap$, using
$Bp+Ap=p$ and $ABp=BAp$, yields
\begin{align*}
 &\lambda(Bp)+\lambda(Ap)+\beta(BAp,A^2p)\\
 &\hspace{4em}=C\lambda(p)+C\beta(Bp,Ap)+D\lambda(p).
\end{align*}
Hence
\begin{align*}
 &Ce+\lambda(Bp)+\lambda(Ap)+\beta(BAp,A^2p)\\
 &\hspace{4em}=C\bigl(e+\lambda(p)+\beta(Bp,Ap)\bigr)+D\lambda(p)
 \in DZ,
\end{align*}
because $C(DZ)\subseteq DZ$.  Therefore
$\phi(K_\phi)\subseteq K_\phi$.

Let $N$ be any compatible resolving kernel.  By
Lemma~\ref{lem:universalinvisible}, $N\subseteq U_\phi$.  Formula
\eqref{eq:invisiblesubgroupformula} therefore gives $N\subseteq K_\phi$.

Conversely, $K_\phi\subseteq U_\phi$, so
Lemma~\ref{lem:universalinvisible} shows that the quotient by $K_\phi$
separates all Reidemeister classes.  It remains only to compute its index.
The projection of $K_\phi$ to $V$ is $BR_D(\omega)$, while
\[
  K_\phi\cap Z=DZ.
\]
Therefore
\begin{align*}
  [\Gamma_\beta:K_\phi]
  &=[V:BR_D(\omega)]\,[Z:DZ]\\
  &=[V:BV]\,[V:R_D(\omega)]\,[Z:DZ]\\
  &=|\det B|\,|\det D|\,[V:R_D(\omega)],
\end{align*}
which proves \eqref{eq:cohomologicalocformula}.

Finally, every twisted class can first be reduced horizontally modulo $BV$
and then centrally modulo $DZ$.  If two representatives with the same chosen
horizontal coordinate are twisted conjugate, then $Bp=0$, hence $p=0$, and
their central coordinates differ by an element of $DZ$.  Thus the resulting
representatives are unique and
\[
  R(\phi)=[V:BV][Z:DZ]=|\det B|\,|\det D|.
\]
This proves \eqref{eq:centralextensionR} and
\eqref{eq:cohomologicalpenalty}.
\end{proof}

\begin{corollary}[Principal torus bundles]
\label{cor:torusbundleoc}
Let
\[
  \T^m\longrightarrow M_\omega\longrightarrow\T^n
\]
be the principal torus bundle associated with the central extension class
$[\omega]\in H^2(\T^n;\Z^m)$, realized by the corresponding two-step
nilmanifold, and let $F\colon M_\omega\to M_\omega$ be the map induced by
$\phi$.  Under \eqref{eq:finiteRregime}, all Reidemeister classes are essential
and
\begin{equation}
  \boxed{
  N(F)=|\det(I-A)|\,|\det(I-C)|,
  \qquad
  \oc(F)=N(F)[V:R_D(\omega)].}
  \label{eq:torusbundleoc}
\end{equation}
In particular, the excess finite-cover cost depends only on $A$, $C$, and the
reduction of the characteristic class $[\omega]$ modulo $(I-C)Z$; it is
independent of the chosen cocycle $\beta$, section, and correction term
$\lambda$.
\end{corollary}

\begin{proof}
The extension group is torsion-free, finitely generated, and two-step
nilpotent; its Mal'cev quotient is a nilmanifold principal torus bundle with
characteristic class $[\omega]$; see
\cite{DekimpeBook,PalaisStewart,Raghunathan}.  In the finite-Reidemeister
regime the standard nilmanifold Nielsen formula gives
$N(F)=R(\phi)$, so every Reidemeister class is essential.  The result follows
from Theorem~\ref{thm:cohomologicalresolution}.  Although the coordinate
formula for $K_\phi$ uses $\beta$ and $\lambda$, its index is the intrinsic
quantity in \eqref{eq:cohomologicalocformula}.
\end{proof}

\begin{corollary}[Scalar dilations of circle bundles over tori]
\label{cor:circlebundleformula}
Let $Z=\Z$, let $\omega\colon\Lambda^2\Z^n\to\Z$ be represented by an
integral skew-symmetric matrix $W$, and consider the compatible scalar maps
\[
  A=kI_n,
  \qquad
  C=k^2,
  \qquad k\ge2.
\]
Put $\Delta=k^2-1$.  Suppose the skew Smith normal form of $W$ is
\[
  W\sim d_1J\oplus\cdots\oplus d_sJ\oplus0_{n-2s},
  \qquad
  J=\begin{pmatrix}0&1\\-1&0\end{pmatrix},
\]
with $d_i>0$ and $d_i\mid d_{i+1}$.  Then
\begin{align}
  N(F_{k,\omega})
  &=(k-1)^n\Delta,
  \label{eq:circlebundleN}\\
  \oc(F_{k,\omega})
  &=(k-1)^n\Delta
  \prod_{i=1}^s
  \left(\frac{\Delta}{\gcd(\Delta,d_i)}\right)^2.
  \label{eq:circlebundleoc}
\end{align}
Thus
\[
  \boxed{
  \frac{\oc(F_{k,\omega})}{N(F_{k,\omega})}
  =\prod_{i=1}^s
  \left(\frac{\Delta}{\gcd(\Delta,d_i)}\right)^2.}
\]
\end{corollary}

\begin{proof}
The reduced radical is
\[
  R_D(\omega)
  =\{p\in\Z^n:Wp\equiv0\pmod\Delta\}.
\]
For a block $d_iJ$, the image modulo $\Delta$ has order
$(\Delta/\gcd(\Delta,d_i))^2$.  Therefore
\[
  [\Z^n:R_{I-C}(\omega)]
  =\prod_{i=1}^s
  \left(\frac{\Delta}{\gcd(\Delta,d_i)}\right)^2.
\]
Corollary~\ref{cor:torusbundleoc} gives the formulas.
\end{proof}

\subsection{Extremal and heterogeneous characteristic blocks}
\label{subsec:extremalblocks}

The preceding formula separates the contribution of each skew Smith block.
This gives an intrinsic zero-penalty criterion for arbitrary principal torus
bundles and exact extremal criteria in the scalar circle-bundle case.

\begin{corollary}[Zero and maximal resolution penalties]
\label{cor:penaltydichotomy}
Under the hypotheses of Corollary~\ref{cor:torusbundleoc},
\[
  \oc(F)=N(F)
  \quad\Longleftrightarrow\quad
  R_D(\omega)=V
  \quad\Longleftrightarrow\quad
  \omega(V,V)\subseteq DZ.
\]
Equivalently, the finite-cover resolution penalty vanishes exactly when the
reduced characteristic form
\(
  \bar\omega_D\colon\Lambda^2V\to Z/DZ
\)
is zero.

In the scalar circle-bundle setting of
Corollary~\ref{cor:circlebundleformula}, the block \(d_iJ\) contributes the
factor
\[
  \pi_i(\Delta)
  :=\left(\frac{\Delta}{\gcd(\Delta,d_i)}\right)^2
\]
to \(\oc(F_{k,\omega})/N(F_{k,\omega})\).  Consequently,
\begin{align*}
  \oc(F_{k,\omega})=N(F_{k,\omega})
  &\quad\Longleftrightarrow\quad
  \Delta\mid d_i\text{ for every }i,\\
  \frac{\oc(F_{k,\omega})}{N(F_{k,\omega})}=\Delta^{2s}
  &\quad\Longleftrightarrow\quad
  \gcd(\Delta,d_i)=1\text{ for every }i.
\end{align*}
Thus a block gives no excess cost precisely when it vanishes modulo
\(\Delta\), and it gives the maximal possible factor \(\Delta^2\) precisely
when it is invertible modulo \(\Delta\).  Different blocks may lie in
zero, maximal, or intermediate regimes within the same bundle.
\end{corollary}

\begin{proof}
The first equivalence follows immediately from
\[
  \frac{\oc(F)}{N(F)}=[V:R_D(\omega)]
\]
and the definition of \(R_D(\omega)\).  In the scalar case,
Corollary~\ref{cor:circlebundleformula} expresses the penalty as the product
of the factors \(\pi_i(\Delta)\).  Each factor equals \(1\) exactly when
\(\Delta\mid d_i\), and equals its maximum \(\Delta^2\) exactly when
\(d_i\) is a unit modulo \(\Delta\).  The stated criteria follow by taking
the product over the \(s\) nonzero skew blocks.
\end{proof}

\begin{example}[A heterogeneous non-Heisenberg bundle]
\label{ex:heterogeneousbundle}
Let \(n=4\), let the characteristic form have skew Smith matrix
\[
  W=J\oplus 8J,
  \qquad
  J=\begin{pmatrix}0&1\\-1&0\end{pmatrix},
\]
and take the scalar dilation \(k=3\).  Then \(\Delta=k^2-1=8\), and
\[
  R_{I-C}(\omega)
  =\ker\bigl(W\colon\Z^4\to(\Z/8\Z)^4\bigr)
  =8\Z^2\oplus\Z^2.
\]
Hence
\[
  [\Z^4:R_{I-C}(\omega)]=8^2=64.
\]
The first block \(J\) is invertible modulo \(8\) and contributes the maximal
factor \(8^2\), whereas the integral block \(8J\) is zero modulo \(8\) and
contributes no excess cost.  Since
\[
  N(F_{3,\omega})=(3-1)^4(3^2-1)=2^4\cdot8=128,
\]
Corollary~\ref{cor:circlebundleformula} gives
\[
  \boxed{\oc(F_{3,\omega})=128\cdot64=8192.}
\]
Thus the characteristic class is neither globally invisible nor uniformly
nondegenerate modulo the central Reidemeister lattice: its two symplectic
blocks occupy opposite extremal regimes.  In particular, the general
cohomological theorem is not tailored to the unimodular Heisenberg form.
\end{example}

\begin{remark}[Topological meaning of the penalty]
\label{rem:topologicalpenalty}
For the trivial bundle $[\omega]=0$, one has
$R_D(\omega)=V$ and hence $\oc(F)=N(F)$.  For a nontrivial bundle, the
quotient $V/R_D(\omega)$ measures the part of the degree-two characteristic
class that remains nondegenerate after reduction modulo the central
Reidemeister lattice.  Thus the factor $[V:R_D(\omega)]$ is not an
additional linear fixed-point count: it is the exact finite-cover price of the
bundle's central twisting.
\end{remark}

\section{Exact nonabelian resolution on Heisenberg nilmanifolds}
\label{sec:heisenberg}

We now specialize the cohomological resolution theorem to its basic
nonabelian symplectic model.  The Heisenberg family makes the characteristic
class explicit and produces an unbounded gap between the Nielsen number and
the finite-cover resolution cost.
For $r\ge1$, write the integral Heisenberg group as
\[
  \Heis_r(\Z)=\Z^r\times\Z^r\times\Z
\]
with multiplication
\[
  (u,v,c)(u',v',c')
  =\bigl(u+u',v+v',c+c'-u'\!\cdot v\bigr).
\]
Its center is $\{(0,0,c):c\in\Z\}$, and
$\Heis_r(\Z)$ is a lattice in the real Heisenberg group.  Let
\[
  M_r=\Heis_r(\Z)\backslash\Heis_r(\R)
\]
be the corresponding compact nilmanifold.

For an integer $k\ge2$, consider the scalar Heisenberg dilation
\[
  \phi_{k,r}(u,v,c)=(ku,kv,k^2c).
\]
It preserves the lattice and therefore induces a self-cover
\[
  F_{k,r}\colon M_r\longrightarrow M_r.
\]
Set
\[
  d=k-1,
  \qquad
  D=k^2-1,
\]
and let
\[
  \Phi_{k,r}=\operatorname{diag}(kI_{2r},k^2)
\]
denote the induced endomorphism of the Heisenberg Lie algebra.

\begin{lemma}[Twisted coordinates and canonical classes]
\label{lem:heisclasses}
For $h=(p,q,w)$ and $g=(a,b,c)$ in $\Heis_r(\Z)$,
\[
 h g\phi_{k,r}(h)^{-1}
 =\bigl(a-dp,b-dq,
 c-a\!\cdot q+k b\!\cdot p-kd\,p\!\cdot q-Dw\bigr).
\]
Every $\phi_{k,r}$-Reidemeister class has a unique representative
\[
  (a,b,c),
  \qquad
  0\le a_i,b_i<d,
  \quad
  0\le c<D.
\]
Consequently,
\[
  R(\phi_{k,r})=d^{2r}D.
\]
\end{lemma}

\begin{proof}
The inverse in the chosen coordinates is
\[
  (u,v,c)^{-1}=(-u,-v,-c-u\!\cdot v),
\]
and direct substitution gives the displayed twisted-conjugation formula.
The first two coordinates may be reduced independently modulo $d$, after
which the last coordinate may be reduced modulo $D$.  If two representatives
in the indicated ranges are twisted conjugate, the first two coordinates
force $p=q=0$, and the last coordinate then forces equality modulo $D$.
The chosen ranges therefore give existence and uniqueness.
\end{proof}

The standard Nielsen formula for affine maps on nilmanifolds gives
\[
  N(F_{k,r})
  =\left|\det\bigl(I-\Phi_{k,r}\bigr)\right|
  =d^{2r}D;
\]
see, for example, \cite{DeconinckDekimpe}.  Thus all Reidemeister classes in
Lemma~\ref{lem:heisclasses} are essential.

\begin{theorem}[Exact Heisenberg observer complexity]
\label{thm:heisexact}
For every $r\ge1$ and $k\ge2$,
\[
  \boxed{\oc(F_{k,r})=d^{2r}D^{2r+1}.}
\]
Equivalently,
\[
  \frac{\oc(F_{k,r})}{N(F_{k,r})}=D^{2r}=(k^2-1)^{2r}.
\]
Moreover, $\dc(F_{k,r})=1$.
\end{theorem}

\begin{proof}
The Heisenberg extension has
\[
  V=\Z^{2r},
  \qquad
  Z=\Z,
\]
and may be written in the form of Section~\ref{sec:cohomological} with
\[
  \beta\bigl((u,v),(u',v')\bigr)=-u'\!\cdot v.
\]
Its commutator form is the standard unimodular symplectic form
\[
  \omega\bigl((u,v),(u',v')\bigr)
  =u\!\cdot v'-u'\!\cdot v.
\]
For the scalar dilation,
\[
  A=kI_{2r},
  \qquad
  C=k^2,
  \qquad
  \lambda=0.
\]
Since
\[
  (I-C)Z=(1-k^2)\Z=(k^2-1)\Z=D\Z
\]
as subgroups of $\Z$, and since $\omega$ is unimodular,
\[
  R_D(\omega)=D\Z^{2r}.
\]
Consequently,
\[
  [V:R_D(\omega)]=D^{2r}.
\]
Corollary~\ref{cor:torusbundleoc} therefore gives
\[
  \oc(F_{k,r})
  =d^{2r}D\cdot D^{2r}
  =d^{2r}D^{2r+1}.
\]
The unique maximal resolving kernel from
Theorem~\ref{thm:cohomologicalresolution} is, in the coordinates of
Lemma~\ref{lem:heisclasses},
\[
  K_{k,r}=dD\Z^r\times dD\Z^r\times D\Z,
\]
which recovers the sharp congruence cover directly.  Finally,
$L(F_{k,r})=\det(I-\Phi_{k,r})\ne0$, so
Proposition~\ref{prop:threelevels} gives $\dc(F_{k,r})=1$.
\end{proof}

\begin{theorem}[All-or-nothing visibility for $k=2$]
\label{thm:heisprofilek2}
For every $r\ge1$, the complete visibility and index-mass profiles of the
scalar Heisenberg dilation $F_{2,r}$ are
\[
  \boxed{
  \VN_{F_{2,r}}(B)=
  \begin{cases}
    1,&B<3^{2r+1},\\[1mm]
    3,&B\ge3^{2r+1},
  \end{cases}}
\]
and
\[
  \MN_{F_{2,r}}(B)=3
  \qquad(B\ge1).
\]
Thus no finite observer sees exactly two of the three essential Nielsen
classes: the profile jumps directly from one visible atom to full resolution.
In particular,
\[
  \dc(F_{2,r})=\cfc(F_{2,r})=1,
  \qquad
  \oc(F_{2,r})=3^{2r+1}.
\]
\end{theorem}

\begin{proof}
Write $\Gamma=\Heis_r(\Z)$ and let
$N\triangleleft\Gamma$ be any finite observer, with quotient map
$q\colon\Gamma\twoheadrightarrow Q$ and induced endomorphism $\bar\phi$.
For $k=2$, Lemma~\ref{lem:heisclasses} gives exactly the three classes
represented by
\[
  z^c:=(0,0,c),\qquad c\in\Z/3\Z.
\]
Since $R(\phi_{2,r})=N(F_{2,r})=3$ and every fixed point is
nondegenerate with local index
$\operatorname{sgn}\det(I-\Phi_{2,r})=-1$, each Reidemeister class
contains exactly one fixed point and all three coefficients of the
Reidemeister trace are $-1$.  The map
$\RN(\phi_{2,r})\to\RN(\bar\phi)$ is surjective, and every quotient class
therefore receives a strictly negative coefficient.  Hence there is no
cancellation and
\[
  V_N(F_{2,r})=R(\bar\phi),
  \qquad
  M_N(F_{2,r})=3.
\]

It remains to determine how the three central representatives merge.  Put
\[
  T_N:=\left\langle
  3,\ e-u\!\cdot v:\ (u,v,e)\in N
  \right\rangle\le\Z.
\]
We claim that
\[
  q(z^c)\sim_{\bar\phi}q(z^{c'})
  \quad\Longleftrightarrow\quad
  c-c'\in T_N.
  \tag{\(*\)}
\]
Indeed, for $h=(p,s,w)$ the twisted-conjugation formula gives
\[
  h z^c\phi_{2,r}(h)^{-1}z^{-c'}
  =(-p,-s,c-c'-2p\!\cdot s-3w).
\]
Thus a single twisted conjugacy in the quotient produces, after writing
$(u,v,e)=(-p,-s,c-c'-2p\!\cdot s-3w)\in N$, the congruence
\[
  c-c'\equiv e-u\!\cdot v\pmod3.
\]
This proves the forward implication in $(*)$.  Conversely, for every
$(u,v,e)\in N$ and every $w\in\Z$, choosing $h=(-u,-v,w)$ realizes the
difference
\[
  e+2u\!\cdot v+3w.
\]
Taking $w=-u\!\cdot v$ realizes the generator $e-u\!\cdot v$, while
taking $(u,v,e)=(0,0,0)$ and $w=1$ realizes the difference $3$.  Finally,
the set of realizable central differences is a subgroup of $\Z$: the
relation is translation invariant on the central representatives, and
symmetry and transitivity give inverses and sums.  It therefore contains
$T_N$, while the forward implication gives the reverse inclusion.  This
proves $(*)$.

Consequently the quotient Reidemeister classes are parametrized by
$\Z/T_N$, and
\[
  V_N(F_{2,r})=[\Z:T_N].
\]
Since $3\Z\subseteq T_N$, this index is either $1$ or $3$.  It is equal to
$3$ exactly when the observer resolves all three original classes.  By
Theorem~\ref{thm:heisexact}, the minimum degree of such an observer is
$3^{2r+1}$.  Below that degree every observer therefore has visible support
one, while the canonical resolving observer attains three at the threshold.
This proves the two profiles and all stated complexities.
\end{proof}

\begin{proposition}[A nondivisorial Heisenberg visibility value]
\label{prop:heisfive}
For the three-dimensional dilation $F_{3,1}$,
\[
  \boxed{\VN_{F_{3,1}}(8)=5.}
\]
Since $N(F_{3,1})=32$, the visible class count need not divide the Nielsen
number.  In particular, nonabelian Heisenberg profiles are not toral divisor
staircases.
\end{proposition}

\begin{proof}
As above, all Reidemeister classes have the same index sign, and the quotient
map on Reidemeister classes is surjective.  Thus, for every finite observer
$N$ with quotient $Q$,
\[
  V_N(F_{3,1})=R(\bar\phi_N).
\]

For the lower bound, take
\[
  N_{2,2}=2\Z\times2\Z\times2\Z\triangleleft\Heis_1(\Z).
\]
The group law and commutator formula show that $N_{2,2}$ is normal; it is
$\phi_{3,1}$-invariant and has index $8$.  The quotient is the dihedral
group $D_8$ of order eight: the images of the two horizontal generators are
involutions and their product has order four.  Since $3\equiv1\pmod2$ and
$9\equiv1\pmod2$, the induced endomorphism is the identity.  The group $D_8$
has five conjugacy classes, so
\[
  \VN_{F_{3,1}}(8)\ge5.
\]

For the reverse inequality, let $Q$ be any compatible quotient of order at
most eight.  It is a two-generated nilpotent group of class at most two.  If
$Q$ is abelian, then the central generator $z=[x,y]$ maps to the identity.
Thus $Q$ is generated by the images of the two horizontal generators and
$\bar\phi_N$ acts as multiplication by $3$.  Twisted-conjugacy classes are
therefore the cosets of $(1-\bar\phi_N)Q=2Q$, and
\[
  R(\bar\phi_N)=|Q/2Q|\le4,
\]
because $Q$ is generated by at most two elements.  If $Q$ is nonabelian,
then necessarily $|Q|=8$ and $Q\cong D_8$ or $Q_8$.  For a generating pair
of $D_8$, either both generators are reflections, in which case the cube map
fixes them, or one is a rotation and the other a reflection, in which case
the cube map is conjugation by that reflection.  For a generating pair
$x,y$ of $Q_8$, conjugation by $xy$ sends both $x$ and $y$ to their inverses,
which are their cubes.  Hence in either nonabelian case $\bar\phi_N$ is an
inner automorphism.  Twisted conjugacy for an inner automorphism is in
bijection with ordinary conjugacy, and both $D_8$ and $Q_8$ have five
conjugacy classes.  Therefore $R(\bar\phi_N)\le5$ for every observer of order
at most eight, proving equality.
\end{proof}

\begin{corollary}[Exact periodic resolution]
\label{cor:periodicresolution}
For every $n\ge1$,
\[
  \oc(F_{k,r}^n)
  =(k^n-1)^{2r}(k^{2n}-1)^{2r+1}.
\]
The minimum resolving cover is the congruence cover associated with
\[
  K_{k^n,r}
  =(k^n-1)(k^{2n}-1)\Z^r
  \times (k^n-1)(k^{2n}-1)\Z^r
  \times (k^{2n}-1)\Z.
\]
\end{corollary}

\begin{proof}
Apply Theorem~\ref{thm:heisexact} to $F_{k,r}^n=F_{k^n,r}$.
\end{proof}

\begin{remark}[Geometric meaning of the Heisenberg lower bound]
The canonical representatives consist of $d^{2r}$ horizontal fixed-point
classes and $D$ central layers over each horizontal class.  The proof shows
that a resolving cover must refine every horizontal lattice direction from
scale $d$ to scale $dD$ and must simultaneously retain a central quotient of
order $D$.  In the language of Theorem~\ref{thm:coverlift}, the associated
lifts remain deck-conjugate in every smaller compatible cover.  The extra
factor $D^{2r}$ in $\oc/N$ is therefore the exact finite-cover cost of
separating the central fixed-point geometry created by the nonabelian
commutator form.
\end{remark}

\begin{corollary}[Nonabelian resolution is necessary]
\label{cor:abelianblind}
Every finite observer of $F_{k,r}$ with abelian quotient identifies the
$D$ central Reidemeister layers over each horizontal class.  In particular,
\[
  V_N(F_{k,r})\le d^{2r}<N(F_{k,r}),
\]
so no abelian observer resolves the spectrum.
\end{corollary}

\begin{proof}
An abelian quotient kills the commutator subgroup, which is the center of
$\Heis_r(\Z)$.  The canonical representatives in
Lemma~\ref{lem:heisclasses} that differ only in their final coordinate
therefore have the same image.
\end{proof}

\begin{proposition}[Heisenberg fibre stacks]
\label{prop:heisenbergfixedpicture}
Represent the three-dimensional Heisenberg nilmanifold as
\[
  M_1=\Heis_1(\Z)\backslash\Heis_1(\R),
\]
with coordinates $[u,v,c]$, where $(u,v)$ is the coordinate on the base
torus and $c$ is the central fibre coordinate.  A point is fixed by the
scalar dilation $F_{k,1}$ if and only if
\[
  (k-1)u\in\Z,
  \qquad
  (k-1)v\in\Z,
\]
and
\[
  (k^2-1)c+(k-1)uv\in\Z.
\]

For $k=2$, the induced map on the base torus has the single fixed point
$(0,0)$, and
\[
  \Fix(F_{2,1})
  =\left\{
  [0,0,0],
  \left[0,0,\frac13\right],
  \left[0,0,\frac23\right]
  \right\}.
\]
Thus all three fixed points lie on one central fibre.  Each is
nondegenerate, has local fixed-point index $-1$, and represents a distinct
essential Nielsen class.  Their pairwise differences lie in the centre
$[\Heis_1(\Z),\Heis_1(\Z)]$.  Consequently, every finite observer with
abelian quotient assigns the same observed deck-conjugacy label to all three
classes.  They become simultaneously distinguishable in the canonical
Heisenberg congruence cover of degree $27$, and no compatible cover of
smaller degree resolves all three:
\[
  N(F_{2,1})=3,
  \qquad
  \oc(F_{2,1})=27.
\]

For $k=3$, the base map has four fixed points
\[
  (u,v)=\left(\frac a2,\frac b2\right),
  \qquad a,b\in\{0,1\}.
\]
Above each of them lie eight equally spaced fixed points,
\[
  c=\frac{m-\frac{ab}{2}}8,
  \qquad m=0,\ldots,7.
\]
Hence
\[
  \#\Fix(F_{3,1})=N(F_{3,1})=4\cdot8=32,
  \qquad
  \oc(F_{3,1})=2048.
\]
\end{proposition}

\begin{figure}[H]
\centering
\includegraphics[width=0.98\textwidth]{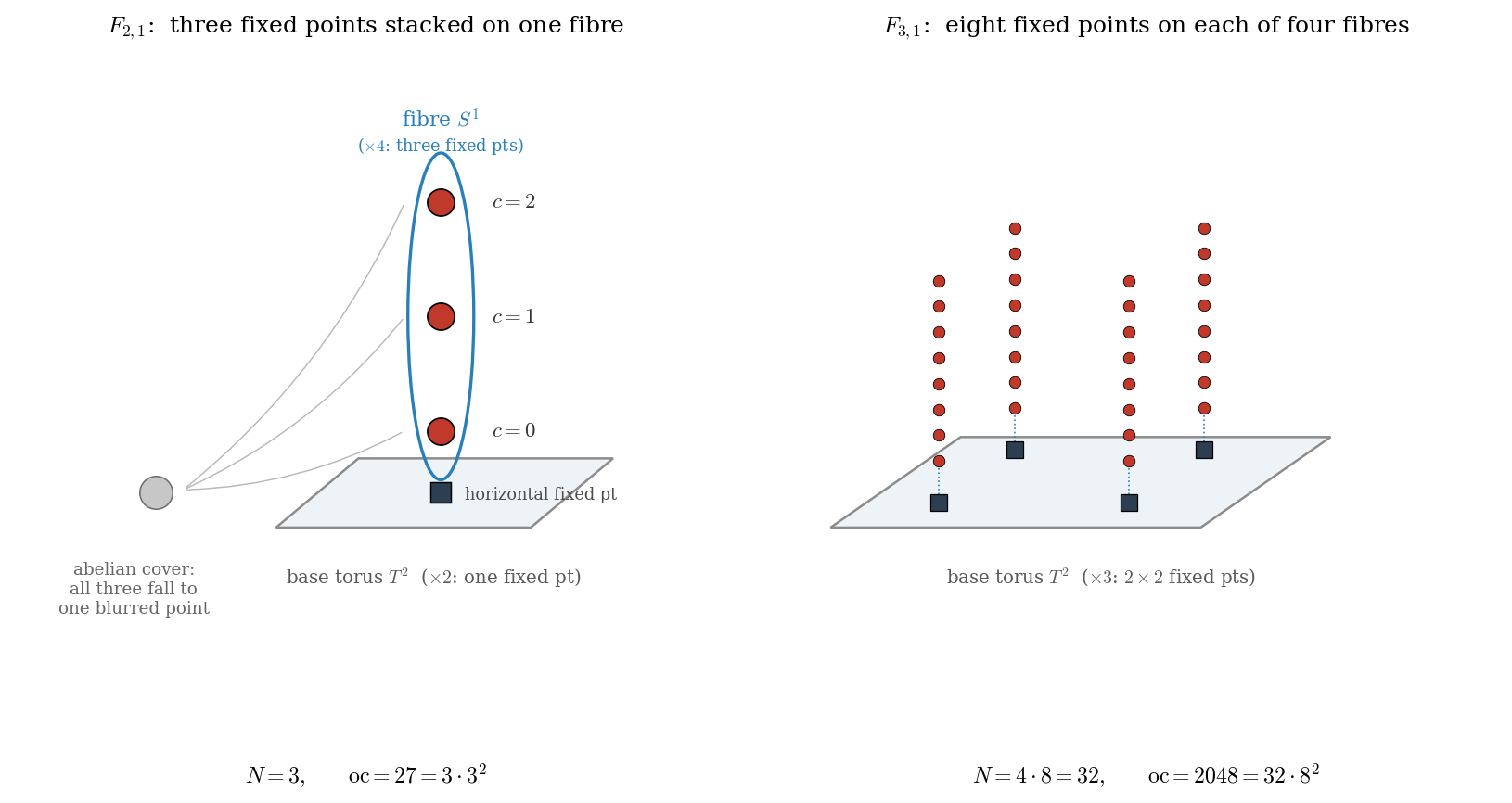}
\caption{Fixed points of the scalar Heisenberg dilations $F_{2,1}$ and
$F_{3,1}$.  For $k=2$, three fixed points are stacked on one central fibre.
For $k=3$, four fixed points occur on the base torus, and eight fixed points
lie above each of them.  The fibre stacks are schematic coordinate pictures,
not a global Euclidean embedding of the nontrivial Heisenberg bundle.}
\label{fig:heisenbergfixedpoints}
\end{figure}

\begin{remark}[Periodic fibre stacks]
Put $d_n=k^n-1$ and $D_n=k^{2n}-1$.  The map $F_{k,r}^n$ has
$d_n^{2r}$ fixed points on the base torus and $D_n$ fixed points on the
central circle above each of them, so
\[
  \#\Fix(F_{k,r}^n)=N(F_{k,r}^n)=d_n^{2r}D_n.
\]
Corollary~\ref{cor:periodicresolution} shows that resolving these classes
requires the larger degree $d_n^{2r}D_n^{2r+1}$.  The factor $D_n^{2r}$ is
therefore a separation cost created by the central bundle twisting, not an
additional fixed-point count.
\end{remark}

\begin{corollary}[Abelian--nilpotent contrast]
\label{cor:contrast}
Let $T_{k,r}$ be the toral endomorphism of $\T^{2r+1}$ induced by
\[
  A_{k,r}=\operatorname{diag}(kI_{2r},k^2).
\]
Then $T_{k,r}$ and $F_{k,r}$ have the same linear eigenvalues, Nielsen number,
and topological entropy:
\[
  N(T_{k,r})=N(F_{k,r})=d^{2r}D,
\]
\[
  h_{\mathrm{top}}(T_{k,r})
  =h_{\mathrm{top}}(F_{k,r})
  =(2r+2)\log k.
\]
Nevertheless,
\[
  \oc(T_{k,r})=d^{2r}D,
  \qquad
  \oc(F_{k,r})=d^{2r}D^{2r+1},
\]
and hence
\[
  \frac{\oc(F_{k,r})}{\oc(T_{k,r})}=D^{2r}.
\]
\end{corollary}

\begin{proof}
The Nielsen-number identities follow from the determinant formula in the
toral and nilmanifold cases.  The toral observer-complexity formula follows
from Theorem~\ref{thm:toralprofile}, while the Heisenberg formula is
Theorem~\ref{thm:heisexact}.  Both maps are expanding coverings of degree
$k^{2r+2}$; hence their topological entropy is $(2r+2)\log k$, by the standard
entropy formula for expanding homogeneous endomorphisms \cite{Bowen}.
\end{proof}

\begin{corollary}[Nonabelian observer entropy]
\label{cor:heisentropy}
For the scalar Heisenberg dilation,
\[
  h_{\mathrm{oc}}(F_{k,r})=(6r+2)\log k
\]
and therefore
\[
  \boxed{
  h_{\mathrm{oc}}(F_{k,r})
  =\frac{3r+1}{r+1}\,h_{\mathrm{top}}(F_{k,r}).}
\]
In particular $h_{\mathrm{oc}}(F_{k,r})>h_{\mathrm{top}}(F_{k,r})$.
\end{corollary}

\begin{proof}
The $n$-th iterate of $F_{k,r}$ is $F_{k^n,r}$.  Corollary~\ref{cor:periodicresolution} gives
\[
  \oc(F_{k,r}^n)
  =(k^n-1)^{2r}(k^{2n}-1)^{2r+1}.
\]
Taking logarithms, dividing by $n$, and passing to the limit yields
\[
  h_{\mathrm{oc}}(F_{k,r})
  =2r\log k+2(2r+1)\log k
  =(6r+2)\log k.
\]
Corollary~\ref{cor:contrast} supplies the topological entropy.
\end{proof}

\begin{theorem}[Exact Heisenberg observer-complexity zeta function]
\label{thm:heiszeta}
Let
\[
  P_r(t):=(t-1)^{2r}(t^2-1)^{2r+1}
  =\sum_{j=0}^{6r+2}c_{r,j}t^j.
\]
For every $r\ge1$ and $k\ge2$,
\[
  \boxed{
  \zetaoc(F_{k,r};z)
  =\prod_{j=0}^{6r+2}(1-k^jz)^{-c_{r,j}}.}
\]
Hence $\zetaoc(F_{k,r};z)$ is rational.  Its Taylor series at the origin has
radius of convergence
\[
  \boxed{\rho_{\mathrm{oc}}(F_{k,r})=k^{-(6r+2)}},
\]
and therefore
\[
  -\log\rho_{\mathrm{oc}}(F_{k,r})
  =h_{\mathrm{oc}}(F_{k,r}).
\]
For $r\ge1$, this zeta function is different from the Nielsen zeta function
of $F_{k,r}$.
\end{theorem}

\begin{proof}
Since $F_{k,r}^n=F_{k^n,r}$, Theorem~\ref{thm:heisexact} gives
\[
  \oc(F_{k,r}^n)
  =(k^n-1)^{2r}(k^{2n}-1)^{2r+1}
  =P_r(k^n)
  =\sum_{j=0}^{6r+2}c_{r,j}k^{jn}.
\]
Using the formal identity
\[
  \sum_{n\ge1}\frac{(k^jz)^n}{n}=-\log(1-k^jz),
\]
we obtain
\[
\begin{aligned}
  \zetaoc(F_{k,r};z)
  &=\exp\left(
      \sum_{j=0}^{6r+2}c_{r,j}
      \sum_{n\ge1}\frac{(k^jz)^n}{n}
    \right)\\
  &=\prod_{j=0}^{6r+2}(1-k^jz)^{-c_{r,j}}.
\end{aligned}
\]
The leading coefficient of $P_r$ is $c_{r,6r+2}=1$.  Hence the product has a
pole at $z=k^{-(6r+2)}$, and every other possible zero or pole has modulus at
least $k^{-(6r+1)}$.  This proves the radius formula.  The logarithmic radius
identity is Corollary~\ref{cor:heisentropy}.  Finally,
\[
  N(F_{k,r}^n)=(k^n-1)^{2r}(k^{2n}-1),
\]
whereas the observer-complexity sequence contains the additional factor
$(k^{2n}-1)^{2r}$; thus the two zeta functions are distinct.
\end{proof}

\begin{theorem}[Classical periodic fixed-point data do not determine finite-cover resolution]
\label{thm:nondetermination}
For every $r\ge1$ and $k\ge2$, let $F_{k,r}$ be the scalar dilation of the
$(2r+1)$-dimensional integral Heisenberg nilmanifold and let $T_{k,r}$ be the
toral endomorphism induced by
\[
  A_{k,r}=\operatorname{diag}(kI_{2r},k^2).
\]
For $n\ge1$, put
\[
  d_n:=k^n-1,
  \qquad
  D_n:=k^{2n}-1.
\]
Then, for every $n\ge1$,
\[
  N(F_{k,r}^n)=N(T_{k,r}^n)=d_n^{2r}D_n
\]
and
\[
  L(F_{k,r}^n)=L(T_{k,r}^n)=-d_n^{2r}D_n.
\]
Consequently, their Nielsen zeta functions
\[
  \zeta_{N,h}(z)
  :=\exp\left(\sum_{n\ge1}\frac{N(h^n)}{n}z^n\right)
\]
coincide:
\[
  \zeta_{N,F_{k,r}}(z)=\zeta_{N,T_{k,r}}(z).
\]
Moreover, the two linear models have the same eigenvalue multiset
\[
  \{\underbrace{k,\ldots,k}_{2r\text{ times}},k^2\},
\]
and
\[
  h_{\mathrm{top}}(F_{k,r})
  =h_{\mathrm{top}}(T_{k,r})
  =(2r+2)\log k.
\]
Nevertheless,
\[
  \oc(T_{k,r}^n)=d_n^{2r}D_n,
  \qquad
  \oc(F_{k,r}^n)=d_n^{2r}D_n^{2r+1},
\]
so
\[
  \boxed{
  \frac{\oc(F_{k,r}^n)}{\oc(T_{k,r}^n)}
  =D_n^{2r}
  =(k^{2n}-1)^{2r}.}
\]
In particular,
\[
  \zetaoc(F_{k,r};z)\ne\zetaoc(T_{k,r};z)
\]
and
\[
  \boxed{
  \lim_{n\to\infty}\frac1n
  \log\frac{\oc(F_{k,r}^n)}{\oc(T_{k,r}^n)}
  =4r\log k>0.}
\]
Thus the finite-cover resolution structure of the essential periodic
Nielsen classes is not determined by the Nielsen sequence or Nielsen zeta
function, the Lefschetz sequence, the common differential eigenvalue multiset, or
topological entropy.
\end{theorem}

\begin{proof}
The $n$-th iterates satisfy
\[
  F_{k,r}^n=F_{k^n,r},
  \qquad
  A_{k,r}^n=\operatorname{diag}(k^nI_{2r},k^{2n}).
\]
The determinant formulas for toral endomorphisms and affine nilmanifold maps
give
\[
  N(F_{k,r}^n)=N(T_{k,r}^n)
  =\left|(1-k^n)^{2r}(1-k^{2n})\right|
  =d_n^{2r}D_n
\]
and
\[
  L(F_{k,r}^n)=L(T_{k,r}^n)
  =(1-k^n)^{2r}(1-k^{2n})
  =-d_n^{2r}D_n.
\]
Equality of the Nielsen zeta functions follows coefficientwise.  The entropy
identity is Corollary~\ref{cor:contrast}.  Applying
Theorem~\ref{thm:toralprofile} to $A_{k,r}^n$ and
Theorem~\ref{thm:heisexact} with $k$ replaced by $k^n$ yields the two
observer-complexity formulas.  Their ratio is $D_n^{2r}$, so the observer zeta
functions differ.  Finally,
\[
  \frac1n\log D_n^{2r}
  =\frac{2r}{n}\log(k^{2n}-1)
  \longrightarrow 4r\log k.
\]
\end{proof}

\begin{remark}
Classical Nielsen and Reidemeister zeta functions package the numbers of
essential fixed-point classes of the iterates; see
\cite{FelHill,DekimpeDugardein,FelLee}.  Theorem~\ref{thm:heiszeta} instead packages the minimum finite-cover degree needed to
resolve those classes simultaneously.  The rationality proof is elementary
because the exact nonabelian complexity formula turns this new sequence into
a finite integral combination of geometric sequences.
\end{remark}

\begin{example}[The three-dimensional dilation]
For $r=1$ and $k=2$,
\[
  N(F_{2,1})=3,
  \qquad
  \oc(F_{2,1})=27.
\]
Thus a natural self-cover of the three-dimensional Heisenberg nilmanifold has
only three essential Nielsen classes, but every finite regular cover that
resolves them has degree at least $27$.
\end{example}

\begin{example}[A Heisenberg automorphism with $\oc=16$]
Let
\[
  \Gamma=\langle x,y,z\mid[x,y]=z,\ [x,z]=[y,z]=1\rangle
\]
and define $\alpha\in\Aut(\Gamma)$ by
\[
  \alpha(x)=x^2y,\qquad \alpha(y)=x,\qquad \alpha(z)=z^{-1}.
\]
On the horizontal lattice and centre it induces
\[
  A=\begin{pmatrix}2&1\\1&0\end{pmatrix},\qquad C=-1.
\]
The commutator form is unimodular symplectic, so for $D=I-C=2$ one has
$R_D(\omega)=2\Z^2$.  Corollary~\ref{cor:torusbundleoc} therefore gives
\[
  N(f_\alpha)=|\det(I-A)|\,|1-C|=4,
  \qquad
  \oc(f_\alpha)=4[\Z^2:2\Z^2]=16.
\]
Every abelian observer kills the centre and hence merges the two central
layers over each horizontal class; none resolves the spectrum.
\end{example}

\section{Conclusion}

Finite compatible covers provide a quantitative resolution theory for the
Reidemeister trace.  The resulting profile distinguishes detection, removal
of index cancellation, and full separation of essential Nielsen classes.
Profinite reconstruction identifies the information invisible to every
finite observer.  The general realization theorem shows that nonzero blind
traces genuinely occur on finite complexes, even for residually finite
fundamental groups, while the delayed-visibility examples show that the
minimum resolving degree is independent of the usual fixed-point counts and
induced homological data.

The central result is the cohomological formula
\[
  \oc(F)=N(F)[\Z^n:R_{I-C}(\omega)]
\]
for compatible maps of principal torus bundles in the finite-Reidemeister
regime.  It identifies the exact
excess resolution cost with the radical index of the reduced degree-two
characteristic class.  The resulting extremal criteria characterize exactly
when the penalty vanishes and, for scalar circle bundles, when every active
skew block contributes maximally.  The heterogeneous example
$W=J\oplus8J$ shows that zero- and maximal-penalty blocks may coexist in one
non-Heisenberg bundle.  The toral case has zero penalty and an exact divisor
profile.  The Heisenberg case has a nondegenerate symplectic class and an
unbounded penalty; already for $k=2$ its full profile is all-or-nothing, and
for $F_{3,1}$ a visible count of five does not divide the Nielsen number
$32$.  At the level of iterates, this distinction produces a
rational observer-complexity zeta function and an exponential gap that is not
determined by the Nielsen or Lefschetz sequences, the Nielsen zeta function,
the differential eigenvalues, or topological entropy.

\end{document}